\documentclass[runningheads]{llncs}
\usepackage[utf8]{inputenc}
\usepackage{algorithm}
\usepackage{algorithmic}

\usepackage[utf8]{inputenc} 
\usepackage[T1]{fontenc}    
\usepackage[english]{babel}

\usepackage{amsfonts}
\usepackage{nicefrac}       
\usepackage{microtype}      
\usepackage{lmodern}
\usepackage{amssymb,amsmath}
\usepackage{relsize}
\usepackage{array}
\usepackage{stmaryrd}
\usepackage{bbm,bm}
\usepackage{subfig}
\usepackage{latexsym}
\usepackage{xcolor}
\usepackage{graphicx}
\usepackage{booktabs}
\usepackage{url}
\usepackage{hyperref}
\hypersetup{
	colorlinks=true,
	linkcolor=blue,
	filecolor=blue,
	anchorcolor=blue,
	urlcolor=cyan,
	citecolor=purple
}
\usepackage{enumerate}
\usepackage[shortlabels]{enumitem}
\usepackage{verbatim}
\usepackage{booktabs}       
\usepackage{multirow}
\usepackage{url}            
\usepackage[numbers]{natbib}

\usepackage[in]{fullpage}

\usepackage[short]{optidef}
\usepackage{algorithm}
\usepackage{thmtools}

\usepackage{empheq}
\usepackage{makecell}

\usepackage{tikz}
\usepackage{tcolorbox}
\usepackage[framemethod=tikz]{mdframed}
 
\declaretheorem{Fact}

\declaretheorem{Assumption}
\DeclareMathOperator{\argmin}{argmin}

\definecolor{gold}{rgb}{0.85,0.65,0}

\numberwithin{subsection}{section}
\numberwithin{equation}{section}
\numberwithin{theorem}{section} 
\numberwithin{lemma}{section} 
\numberwithin{definition}{section} 
\numberwithin{proposition}{section} 
\numberwithin{lemma}{section} 
\numberwithin{corollary}{section} 
\numberwithin{Fact}{section} 
\numberwithin{remark}{section}

\let\emptyset\varnothing

\usepackage{mathrsfs}
\usepackage{marvosym}

\def\B{{\mathbb{B}}}

\def\R{{\mathbb{R}}}

\def\bI{{\mathbf{I}}}

\def\cC{{\cal C}}
\def\cD{{\cal D}}
\def\cE{{\cal E}}

\def\cI{{\cal I}}
\def\cJ{{\cal J}}

\def\cL{{\cal L}}

\def\cO{{\cal O}}
\def\cP{{\cal P}}

\def\cR{{\cal R}}

\def\sP{\mathscr{P}}

\def\e{{\bm e}}

\def\1{{\bm 1}}

\DeclareMathOperator*{\argmax}{arg\,max}

\DeclareMathOperator{\conv}{conv}

\usepackage{soul}

\date{}
\begin{document}
\title{ \textbf{Accelerated Price Adjustment for Fisher Markets\\
			with Exact Recovery of Competitive Equilibrium\thanks{Authors are in $\alpha$-$\beta$ order.} } }
\titlerunning{Accelerated Price Adjustment for Fisher Markets}
	%
	\author{He Chen\inst{1} \and
		Chonghe Jiang\inst{2} \and
		Anthony Man-Cho So\textsuperscript{(\Letter)}\inst{3} }
	\authorrunning{H. Chen et al.}
	%
	\institute{The Chinese University of Hong Kong, Hong Kong SAR, China \\
		\email{hechen@link.cuhk.edu.hk}\\ \and 
		Massachusetts Institute of Technology, Cambridge, MA, US \\
		\email{chonghej@mit.edu}  \\ \and
		The Chinese University of Hong Kong, Hong Kong SAR, China \\
		\email{manchoso@se.cuhk.edu.hk}
	}
\maketitle
\begin{abstract}
The canonical price-adjustment process, t\^atonnement, typically fails to converge to the exact competitive equilibrium (CE) and requires a high iteration complexity of $\tilde{\cO}(1/\epsilon)$ to compute $\epsilon$-CE prices in widely studied linear and quasi-linear Fisher markets. This paper proposes refined price-adjustment processes to overcome these limitations. By formulating the task of finding CE of a (quasi-)linear Fisher market as a strongly convex nonsmooth minimization problem, we develop a novel accelerated price-adjustment method (APM) that finds an $\epsilon$-CE price in $\tilde{\cO}(1/\sqrt{\epsilon})$ lightweight iterations, which significantly improves upon the iteration complexities of t\^atonnement methods. Furthermore, through our new formulation, we construct a recovery oracle that maps approximate CE prices to exact CE prices at a low computational cost. By coupling this recovery oracle with APM, we obtain an adaptive price-adjustment method whose iterates converge to CE prices in finite steps. To the best of our knowledge, this is the first convergence guarantee to exact CE for price-adjustment methods in linear and quasi-linear Fisher markets. Our developments pave the way for efficient lightweight computation of CE prices. We also present numerical results to demonstrate the fast convergence of the proposed methods and the efficient recovery of CE prices. 
\end{abstract}

\section{Introduction}
\textit{Competitive equilibrium} (CE) is one of the most fundamental concepts in economic theory and has many modern applications, such as the digital ad auction, online resource allocation, and fair division; see, e.g., \cite{conitzer2022pacing,balseiro2021regularized,jalota2023fisher,peysakhovich2023implementing}, and the references therein. This concept was initially proposed by \citet{walras2014leon} in the nineteenth century to characterize the ideal status of a general market, in which every agent sells his initial endowment to buy goods.  As a special case of Walras's model, Fisher considered a market that sells $m$ goods to $n$ buyers, where every buyer has a fixed budget \cite{brainard2005compute}. In such a market, a group of allocations and prices is called a CE if
(i) for each buyer, the total price of collected goods does not exceed his budget;
(ii) every buyer gets his optimal bundle; 
(iii) all goods are cleared out. 

The existence of CE was unknown until the work of \citet{arrow1954existence} in the 1950s. Later, \citet{eisenberg1959consensus} showed that the CE of a Fisher market are the Karush-Kuhn-Tucker (KKT) points of a certain convex program. Since then, the computation of CE has received significant attention. Based on the Eisenberg-Gale (EG) program, many polynomial-time algorithms, such as the ellipsoid method \cite{jain2007polynomial}, interior point method \cite{ye2008path}, and combinatorial methods \cite{chaudhury2018combinatorial,devanur2008market,duan2015combinatorial}, were proposed to compute CE. However, these methods adopt a centralized model that precludes parallel computation and requires solving expensive subproblems in each iteration. As markets grow in scale and automated markets prevail in modern applications, there is an increasing need for \textit{lightweight} and \textit{decentralized}\footnote{As opposed to centralized methods, we refer to decentralized methods as algorithms that split the update into subtasks, executed in parallel across multiple nodes.} algorithms. Typical examples of such algorithms include proportional response dynamics \cite{birnbaum2011distributed,wu2007proportional}, first-order methods \cite{gao2020first,liu2025pdhcg}, and t\^atonnement \cite{goktas2023tatonnement,cheung2020tatonnement}.

T\^atonnement, proposed by \citet{walras2014leon}, is one of the earliest lightweight and decentralized algorithms for Fisher markets. It refers to a price-adjustment process that increases the price of a good when the demand exceeds the supply and decreases the price when the demand is smaller \cite{cheung2020tatonnement}. Thanks to its purely price-based update rule, t\^atonnement requires only $\mathcal{O}(m)$ memory to store the iterates,  significantly less than the $\mathcal{O}(mn)$  memory required by other lightweight algorithms, making it appealing in large-scale markets. 
More importantly, as t\^atonnement mimics the price behavior in real-world markets, it has garnered much interest in its theoretical properties, and a long line of work has attempted to establish its convergence.  \citet{walras2014leon} conjectured that t\^atonnement converges to CE. Nearly half a century later, \citet{arrow1958stability} provided a proof for the convergence of continuous-time t\^atonnement under the weak gross substitutes (WGS) assumption. As for the discrete-time t\^atonnement, there are different variants and their convergence properties depend on the specific utility functions \cite{cheung2014analyzing}. \citet{cheung2012tatonnement} showed that for \textit{constant-elasticity-of-substitution} (CES) utilities\footnote{CES utilities are of the form $u_i(y)=(\sum_{j\in[m]}a_{ij}y_j^{\rho_i})^{1/\rho_i}$\text{ for } $y\in\R^m_+$ with $\rho_i\leq1$ and $a_{ij}\geq0$; see, e.g., \cite[Definition 2.5]{cheung2020tatonnement}.} with $0\leq\rho_i<1$, the multiplicative t\^atonnement achieves a linear convergence rate to CE. 
\citet{cheung2020tatonnement} considered entropic t\^atonnement and proved that its iterates converge to CE at a linear rate for complementary CES utilities ($-\infty<\rho_i<0$) and at a sublinear rate for Leontief utilities. Further, in homothetic Fisher markets, \citet{goktas2023tatonnement} established a convergence rate of $\cO((1+E^2)/T)$ for entropic t\^atonnement, where $E$ represents an upper bound on the elasticity of demand and $T$ denotes the iteration number.
However, the above convergence results are not applicable to the most commonly used linear and quasi-linear utilities (characterized by parameters $\rho_i=1$ and $E=+\infty$). Indeed, it has been shown that t\^atonnement does not converge to the exact CE in (quasi-)linear Fisher markets (see \cite[Example 1]{cole2019balancing}), and only weaker convergence guarantees have been established. \citet{cole2019balancing} proved that the entropic t\^atonnement converges to an approximate CE in linear Fisher markets.  \citet{nan2024convergence} showed that for both linear and quasi-linear utilities, the additive t\^atonnement converges to an $\epsilon$-CE at a linear rate of $1-\Theta(\epsilon)$. Consequently, the additive t\^atonnement finds an $\epsilon$-CE in ${\cO}(\log(\frac1{\epsilon})\frac{1}{\epsilon})$ iterations. In comparison, other lightweight algorithms, such as the mirror descent methods, which iteratively adjust bid vectors \cite{gao2020first,birnbaum2011distributed}, converge to the exact CE at a rate of $\cO(1/T)$.  

From the above discussion, we see that with linear and quasi-linear utilities, the dominant price-adjustment processes—t\^atonnement methods—not only fail to converge to a CE but also require a high iteration complexity of $\tilde{\cO}(1/\epsilon)$ for finding an $\epsilon$-CE. These limitations motivate the following questions for (quasi-)linear Fisher markets: 
\begin{enumerate}[itemindent=1em]
\item[\textbf{(Q1)}] \textit{Can we develop a faster price-adjustment process for computing approximate CE prices}?
\item[\textbf{(Q2)}] \textit{Can we refine price-adjustment processes to ensure iterate convergence to exact CE prices}? 
\end{enumerate}

\subsection{Technical Contributions}  In this paper, we develop novel price-adjustment processes to provide affirmative answers to (Q1) and (Q2). Starting from the dual of the EG program, we treat the linear and quasi-linear Fisher markets in a unified manner and formulate the task of computing CE into a box-constrained strongly convex nonsmooth minimization problem. In the new formulation, the objective function merely involves the prices and consists of a sum of exponential and piecewise linear functions. Such a simple structure facilitates the algorithm design. Specifically,
by smoothing those piecewise linear functions, we obtain a surrogate problem whose objective function is strongly convex and smooth with a computable modulus. Therefore, acceleration methods can be applied, leading to the so-called \textit{accelerated price-adjustment method} (APM), a novel price-adjustment process distinct from t\^atonnement: In each iteration, instead of reacting to the present excess supply, APM predicts the future excess supply to adjust the prices. As one of our main contributions, we prove that APM finds $\epsilon$-CE prices in $\tilde{\cO}(1/\sqrt{\epsilon})$ iterations. Such a rate significantly improves upon the iteration complexities of t\^atonnement methods, thereby providing an affirmative answer to (Q1). Furthermore, we show that the prices produced by APM can be used to compute an approximate CE allocation, offering an additional advantage over t\^atonnement methods.

To address (Q2), as our second main contribution, we construct a \textit{recovery oracle} that maps an $\epsilon$-CE price (for sufficiently small $\epsilon$) to an exact CE price and show that it only requires $\cO((m+n)^2)$ arithmetic operations, comparable to the iteration cost of t\^atonnement. This result is significant, as it establishes a direct connection between approximate and exact CE prices in (quasi-)linear Fisher markets and provides a tool for price-adjustment processes to compute exact CE. Subsequently, by coupling the recovery oracle with APM (resp. t\^atonnement methods), we develop the \textit{adaptive} APM (resp. adaptive t\^atonnement), which is guaranteed to find CE prices in finite steps under a practical stopping criterion. To the best of our knowledge, this is the first price-adjustment process with an iterate convergence guarantee to CE prices, providing a definitive answer to (Q2). We note that the techniques developed in this paper, especially the recovery oracle and the acceleration scheme, are novel to the study of computing CE and may also extend to the computation of other market equilibria.

Finally, we demonstrate the practicality and efficiency of the proposed price-adjustment methods via numerical experiments. We compare APM with t\^atonnement, mirror descent, and primal-dual hybrid gradient method (PDHG) \cite{liu2025pdhcg,chambolle2011first} by evaluating their iteration number for finding an approximate CE price. Our numerical results show that, for both synthetic and real-world datasets, APM requires only about a quarter of the iterations of other algorithms. Furthermore, we compare the CPU time of the adaptive APM with the off-the-shelf solver to compute exact CE prices. It turns out that the adaptive APM is $10$-$100$ times faster than the solver. With these competitive numerical performance and strong theoretical guarantees, the (adaptive) APM provides a substantial refinement over existing price-adjustment processes in (quasi-)linear Fisher markets.

\subsection{Further Related Work}
Many different variants of the Fisher market have been considered in the literature, leading to a host of new equilibrium notions. Examples include CE for chores \cite{budish2011combinatorial,bogomolnaia2017competitive,chaudhury2024competitive},  CE for public goods allocation \cite{jalota2023fisher,jalota2020markets}, and pacing equilibrium in auction markets \cite{conitzer2022multiplicative,conitzer2018multiplicative,conitzer2019pacing}. Computing these equilibria amounts to finding solutions to different optimization problems, for which techniques from optimization, e.g., error bound \cite{chen2024computing,nan2024convergence,liu2025pdhcg}, duality \cite{cole2017convex,gao2020first}, and penalty method \cite{boodaghians2022polynomial}, play a central role in obtaining solutions efficiently. 

The optimization tools used in this paper basically lie in the field of first-order methods. Let us review them in order. We start with Nesterov's acceleration, the key tool to the development of our APM. In each iteration, Nesterov's acceleration method performs a gradient descent step with a carefully chosen stepsize and constructs a linear combination of two consecutive iterates. As the optimal first-order method for convex $L$-smooth optimization (see, e.g., \cite[Chapter 2]{nesterov2013introductory}), Nesterov's acceleration method improves the rate of the gradient descent method (GDM) from $\cO(1/k)$ to $\cO(1/k^2)$, where the convergence measure is the function value gap and $k$ denotes the iteration number. If the objective function is  $\sigma$-strongly convex, then it accelerates the linear rate of the GDM by a factor of $\sqrt{{\sigma}/{L}}$. Some other acceleration methods, e.g., the heavy-ball method \cite{nemirovskiy1984iterative} and FISTA \cite{beck2009fast}, achieve the same results; see \cite{dAspremont2021AccelerationMethods} for a survey. 

For all these acceleration methods, the smoothness of the objective function is a necessary condition. However, many real applications induce nonsmoothness in objective functions of the corresponding optimization problems. To deal with nonsmooth functions, a popular approach is to apply the smoothing technique, which involves deriving a smooth approximation of a nonsmooth function through appropriate regularization \cite{nesterov2005smooth,nesterov2007smoothing}. Such an approximation preserves the convexity of the original problem and hence facilitates the application of acceleration methods. It is noteworthy that \citet{chen2024computing} has recently used the smoothing technique to develop an unconstrained smooth approximation problem for computing CE for chores. Nevertheless, they failed to accelerate their algorithm due to the nonconvex nature of the chores setting. In comparison, our work shows that for (quasi-)linear Fisher markets, acceleration can be achieved for price-adjustment processes.

On another front, when it comes to the computation of exact CE, the above tools offer little help. A classical technique to compute an exact CE is the interior point algorithm rounding procedure \cite{ye2008path,mehrotra1993finding}. Assuming that the inputs are rational with bit-length bounded by $\cL$, this procedure rounds an $\epsilon$-KKT point ($\epsilon\leq2^{-\cO(\cL)}$) into an exact CE through a system of linear equalities and inequalities. Then, together with the modified primal-dual path-following algorithm, it outputs an exact CE in at most $\cO(\sqrt{mn}(m+n)^3\cL)$ arithmetic operations; see \cite[Theorem 3]{ye2008path}. However, such a procedure needs both $\epsilon$-CE prices and allocations to solve the required linear system, making it inapplicable to price-adjustment methods that only produce $\epsilon$-CE prices. In contrast, our recovery oracle, based on the local solvability of the nonlinear optimality conditions of the new formulation, only requires approximate CE prices as input. Consequently, it can be coupled with price-adjustment methods to yield finite-step convergence to CE prices. Furthermore, the oracle avoids large linear systems and enjoys a significantly lower complexity of $\cO((m+n)^2)$, consistent with the lightweight nature of price-adjustment methods.

\paragraph{Organization.} This paper is organized as follows. Sec. \ref{sec:CE} introduces the formal definition of CE. Sec. \ref{sec:formulation} proposes a unified box-constrained strongly convex formulation to compute CE for linear and quasi-linear utilities. In Sec. \ref{sec:epsilon-CE}, we develop APM and prove that it finds an $\epsilon$-CE in $\tilde{\cO}(1/\sqrt{\epsilon})$ iterations. In Sec. \ref{sec:exact-CE}, we present the recovery oracle and show how it can be used to recover an exact CE. By coupling the recovery oracle with APM, we further develop the adaptive APM that is guaranteed to find an exact CE in finite steps.  In Sec. \ref{sec:experiments}, we report numerical results to show the superior performance of our algorithms. Finally, we give some closing remarks in Sec. \ref{sec:conclusion}. 

\paragraph{Notation.} The notation used in this paper is mostly standard. We use $[m]$ to denote the set $\{1,\ldots,m\}$ for any positive integer $m$. For an index set $J\subseteq [m]$, we use $|J|$ to denote its cardinality. We use $\e_j$, $j\in[m]$ to denote the standard basis vectors in $\R^m$ and let $\e_0$ be the zero vector of $\R^m$ for the sake of consistency.
Let $\cD_m$ denote the $m$-dimensional simplex, i.e., $\cD_m=\{y\in\R^m_+:\sum_{j\in[m]}y_j=1\}$. For a vector $y\in\R^m$ and a closed convex set $\cC\in\R^m$, let $\Pi_{\cC}(y)$ to denote the projection of $y$ onto $\cC$. We adopt the convention that $\log 0=-\infty$, $\exp({-\infty})=0$, $[0]=\emptyset$, and $\max_{j\in\emptyset}\{y_j\}=-\infty$.

\section{Preliminary: Definition of CE}\label{sec:CE}
Consider a Fisher market with $m$ divisible items and $n$ buyers. Each buyer $i$ spends his budget  $B_i$ ($B_i>0$) purchasing a bundle of goods to maximize his utility $u_i$, which is a function of his allocation vector $x_i\in\R^m_+$. By setting the price vector $p \in \mathbb{R}_{+}^m$ appropriately, the market would strike a balance at the following equilibrium \cite{gao2020first,arrow1954existence}.
\begin{definition}[Competitive Equilibrium in Fisher Market]\label{def:ME} We say that a price vector $p^*\in\R^m_+$ and an allocation $x^*\in\R^{n\times m}_+$ satisfy competitive equilibrium (CE) if and only if
\begin{enumerate}[label={{\rm (E\arabic*).}},itemsep=0.5pt,leftmargin=*]
    \item $\sum_{j\in[m]} p_jx^*_{ij} \leq B_i$ for all $i \in[n]$; 
    \item $u_i(x^*_i)=\underset{x_i \in \R_{+}^{n}}{\operatorname{max}}\left\{u_{i}(x_i):\sum_{j\in[m]} p_jx_{ij} \leq B_i\right\}$ for all $i \in[n]$; 
    \item $\sum_{i\in[n]} x^*_{i j} \leq 1$ for all $j \in [m]$ and the equality must hold if $p_j^{*} > 0$.
\end{enumerate}
\end{definition}
Utility functions $u_i,i\in[n]$ play a crucial role in determining CE as they define the objective functions in the buyer's individual optimization problem (E2). The most simple example of $u_i$ is arguably linear utility, i.e., $u_i:x_i\mapsto \sum _{j\in[m]} v_{ij} x_{ij}$, only parameterized by a utility vector $v_i\in\R^m_+$. Another prevalent example is the quasi-linear utility $u_i:x_i\mapsto \sum _{j\in[m]} (v_{ij} - p_j) x_{ij}$, which represents the utility of the goods purchased minus the total payment.
With this utility, a buyer $i$ would prefer not to purchase any goods if the prices $p_j$ exceed his valuations $v_{ij}$ for all $j\in[m]$. 

Condition (E3) refers to market clearance. To ensure that all goods are sold out at CE, throughout the paper, we make the following assumption that is commonly adopted in Fisher markets; see, e.g., \cite[Assumption 1]{ye2008path} and \cite[Section 4]{gao2020first}. 
\begin{Assumption}
    Without loss of generality, we assume that the utility matrix $v \in \mathbb{R}^{n\times m}$ is nondegenerate, i.e., it does not allow a zero row or column.
    \label{ass:nondegenerate}
\end{Assumption}
Assumption \ref{ass:nondegenerate} ensures that the prices at CE, i.e., $p^*_j,j\in[m]$, are strictly positive (see, e.g., \cite[Sec. 2]{ye2008path} and \cite[Lemma 2]{gao2020first}), and hence $\sum_{i\in[n]} x^*_{i j} = 1$ for all $i\in[n]$ by (E3). Further, we have  upper and lower bounds on $p^*_j,j\in[m]$ by \cite[Lemma 2 and 12]{gao2020first}.
\begin{lemma}\label{le:bound}
For linear and quasi-linear utilities, the prices $p^*_j,j\in[m]$ at CE have upper and lower bounds $\bar{p}\coloneqq (1-\alpha^{-1})\|B\|_1+\max_{i\in[n],j\in[m]} \alpha^{-1} v_{i j}$ and $\underline{p}\coloneqq\min_{j\in[m]}\max_{i\in[n]} \frac{v_{i j} B_i}{\|v_i\|_1+\alpha^{-1}{B_i}}$, i.e.,
\[\underline{p}\leq p^*_j\leq\bar{p},\qquad\forall\ j\in[m],\]
where $\alpha=1$ (resp. $\alpha=+\infty$) corresponds to quasi-linear (resp. linear) utilities. 
\end{lemma}

\section{Unified Strongly Convex Formulation}\label{sec:formulation}
In this section, we formulate the task of computing CE into a structured strongly convex minimization problem. To begin, we consider the dual EG programs with widely studied linear and quasi-linear utilities \cite{cole2017convex,conitzer2022pacing,gao2020first}, whose optimal solution $p^*$ is nothing but CE price vector.
\begin{flalign*}
	\begin{array}{cl} 
		\quad\textit{Linear:  } &  \\
		\min \limits_{p\in\R^m_+}&\sum \limits_{j \in [m]} p_j-\sum \limits_{i \in [n]} B_i \log \left(\beta_i\right)\\ \text { subject to } 
		& p_j \geq v_{i j} \beta_i,\ \forall\ i \in [n], j \in [m];\\
		&\empty\\
	\end{array} \Bigg| &\quad
\begin{array}{cl}
	\textit{Quasi-linear:}& \\
	\min \limits_{p\in\R^m_+}&\sum \limits_{j \in [m]} p_j-\sum \limits_{i \in [n]} B_i \log \left(\beta_i\right)\\ \text { subject to } 
	& p_j \geq v_{i j} \beta_i,\ \forall\ i \in [n], j \in [m]\\
	&\beta_{i} \leq 1,\quad\quad \forall\ i \in [n].
\end{array}      
\end{flalign*}
To investigate the above two optimization problems in a unified manner, we introduce a parameter $\alpha\in\{1,+\infty\}$ and consider the following formulation:
\begin{equation}
\begin{array}{cl}
\min \limits_{p\in\R^m_+}&\sum \limits_{j \in [m]} p_j-\sum \limits_{i \in [n]} B_i \log \left(\beta_i\right)\\ \text { subject to } 
& p_j \geq v_{i j} \beta_i,\quad \forall\ i \in [n], j \in [m]\\
&\beta_{i} \leq \alpha,\qquad\ \forall\ i \in [n],
\end{array}   
\label{eq:Unify1}
\end{equation}
where $\alpha=+\infty$ (resp. $\alpha=1$) corresponds to linear (resp. quasi-linear) Fisher market setting.

Observe that the objective function of Problem \eqref{eq:Unify1} is monotonically decreasing with respect to $\beta_i,i\in[n]$. We can omit all the constraints and obtain a problem that merely concerns $p$.
\begin{equation}
    \min\limits_{p\in\R^m_+}\quad  f(p)\coloneqq\sum \limits_{j\in [m]} p_j-\sum \limits_{i \in [n]} B_i \log \left(\min \left\{\alpha,\frac{p_j}{v_{ij}},j\in[m] \right\}\right).
    \label{eq:unified-Dual}
\end{equation}
By substituting $p_j$ with $\exp({\mu_j})$ for $j\in[m]$ and introducing the notation $v_{i0}=1$ and $\mu_0=\log(\alpha)$, we further simplify the above problem as 
\begin{equation}
        \min \limits_{\mu \in \mathbb{R}^{m}}\quad F(\mu)\coloneqq\sum \limits_{j \in [m]} \exp(\mu_{j})+\sum \limits_{i \in [n]} B_i \max_{j\in\{0\}\cup[m]} \left\{\log(v_{ij})-\mu_{j}\right\}.
        \label{eq:QL-exp}
        \tag{$\sP$}
\end{equation}
With the exponential functions $\exp(\mu_j),j\in[m]$ incorporated, the objective function $F$ is strongly convex on $\R^m$. However, it lacks a global strong convexity modulus. To address this issue, we impose box constraints on $\mu_j,j\in[m]$. Specifically, let $\mu^*$ denote the optimal solution of \eqref{eq:QL-exp} and use the relations $p_j=\exp(\mu_j),j\in[m]$. By Lemma \ref{le:bound}, we see that
\begin{equation}\label{eq:mu}
	\underline{\mu}\coloneqq \log(\underline{p}) \leq\mu^*_j\leq\bar{\mu}\coloneqq \log(\bar{p}),\quad\ \forall\ j\in[m].
\end{equation}
Therefore, we can constrain the variables $\mu_j,j\in[m]$ in the box $[\underline{\mu},\bar{\mu}]$ without changing the optimal solution of \eqref{eq:QL-exp}, leading to the following box-constrained formulation:
\begin{equation}
\begin{array}{cl}
    \min \limits_{\mu \in \mathbb{R}^{m}} &   \quad F(\mu)=\sum \limits_{j \in [m]} \exp(\mu_{j})+\sum \limits_{i \in [n]} B_i \max\limits_{j\in\{0\}\cup[m]} \left\{\log(v_{ij})-\mu_{j}\right\} \\
    \text{ subject to } &\quad\underline{\mu}\leq \mu_j\leq\bar{\mu},\qquad\qquad\qquad\qquad\forall\ j\in[m].
\end{array}
\label{eq:box}
\tag{${\sP}^*$}
\end{equation}
Clearly, in the feasible region of \eqref{eq:box}, the objective function $F$ is strongly convex with modulus $\sigma_0= \exp({\underline{\mu}})$, and its smooth part, i.e., the exponential sum $\sum_{j \in [m]} \exp(\mu_{j})$, has a Lipschitz continuous gradient with modulus $L_0=\exp({\bar{\mu}})$. Moreover, the box constraints allow a simple closed-form projection, where the $j$-th element of the projection of $\mu\in\R^m$ is given by
$\Pi_{[\underline{\mu},\bar{\mu}]}(\mu_j)=\max\{\underline{\mu},\min\left\{\mu_j,\bar{\mu}\right\}\}$. These properties will prove crucial for accelerating price adjustment.
\section{Accelerated Price Adjustment} \label{sec:epsilon-CE}
To tackle the formulation \eqref{eq:box}, we first smooth the $\max$ terms in the objective function $F$ via the entropy regularization. Specifically, we consider an alternative form of $F$, i.e.,
\begin{equation}\label{eq:F_max_form}
    F(\mu)=\sum_{j \in [m]} \exp(\mu_{j})+ \sum_{i\in [n]} B_{i} \max \limits_{\lambda_{i}\in \cD_{m+1}}\left\{\sum \limits_{j \in [m] \cup \{ 0\}} \lambda_{ij} \left(\log\left(v_{ij}\right)-\mu_{j}\right)\right\},
\end{equation}
and approximate it by
\begin{equation}\label{eq:F_delta_max}
\begin{aligned}
     F_{\delta}(\mu) &= \sum\limits_{j \in [m]} \exp(\mu_{j}) + \sum \limits_{i\in [n]} B_{i} \max \limits_{\lambda_{i}\in \cD_{m+1}} \left\{\sum \limits_{j \in  \{ 0\}\cup[m]} \lambda_{ij} \left(\log\left(v_{ij}\right)-\mu_{j}\right)-\delta \lambda_{ij}\log\left(\lambda_{ij}\right)\right\}.
\end{aligned}
\end{equation}
Here, $\delta>0$ is the approximation parameter and the optimal value of the inner maximization problems can be computed by plugging in the optimal solution $\lambda^*_{ij}=\exp({\frac{\log(v_{ij})-\mu_j}{\delta}})/\sum_{j\in\{0\}\cup[m]}\exp({\frac{\log(v_{ij})-\mu_j}{\delta}})$.
Then, by replacing $F$ with $F_{\delta}$ in \eqref{eq:box} and relaxing the box constraints $\mu_j\in[\underline{\mu},\bar{\mu}]$ to $\mu_j\in[\underline{\mu}-\eta,\bar{\mu}+\eta]$, $j\in[m]$ for some $\eta\geq 0$, we obtain the following approximation problem for \eqref{eq:QL-exp}:
\begin{equation}
    \begin{aligned}
     \min \limits_{\mu \in \mathbb{R}^{m}} \quad & F_{\delta}(\mu) =  \sum \limits_{j\in [m]}\exp(\mu_j)+\delta\sum \limits_{i\in [n]}B_{i} \log\left(\sum \limits_{j\in \{0\}\cup[m]}\exp\left({\frac{\log\left(v_{ij}\right)-\mu_j}{\delta}}\right)\right)
     \\
    \text{ subject to } &\quad\underline{\mu}-\eta\leq \mu_j\leq\bar{\mu}+\eta,\qquad\qquad\qquad\qquad\forall\ j\in[m].
    \end{aligned}
    \label{eq:box-smooth}
    \tag{${\sP}_{\delta}$}
\end{equation}
We have the following properties for the approximation problem \eqref{eq:box-smooth}.
\begin{Fact}[Properties of Problem \eqref{eq:box-smooth}]\label{fact:f_delta}
    The following holds:
    \begin{enumerate}[{\rm (i)},itemsep=0.5pt,leftmargin=*]
        \item Smoothness: $F_{\delta}$ is $L$-smooth in the box  $[(\underline{\mu}-\eta)\1_m,(\bar\mu+\eta)\1_m]$, where $L = \exp({\bar{\mu}+\eta})+\|B\|_1/{\delta}$.
        \item Strong convexity: $F_{\delta}$ is $\sigma$-strongly convex on the box constrain set, where $\sigma = \exp({\underline{\mu}-\eta})$.
        \item Approximation error: $F\leq F_{\delta}\leq F+\delta\log(m+1) \|B\|_1.$
        \item Subgradient approximation: $\lim_{\delta\to0}\nabla F_{\delta}(\mu)\in\partial F(\mu)$ for all $\mu\in\R^m$.
    \end{enumerate}
\end{Fact}
Now, we are ready to present our accelerated price-adjustment method (APM), which consists of a gradient step and an acceleration step. The gradient step computes the gradient of the approximation function $F_{\delta}$ and performs the projection onto the box constraints, where the computation cost is $\cO(mn)$. The acceleration step computes a simple linear combination at a cost of $\cO(m)$. Therefore, the total iteration cost of APM is $\cO(mn)$, identical to that of t\^atonnement methods. Furthermore, by selecting an appropriate relaxation radius $\eta$, APM includes a practical stopping criterion to identify an $\epsilon$-CE price. 

\begin{empheq}[box=\fbox]{align*}
   & \textbf{Accelerated Price-adjustment Method (APM)}\quad\\
    \textbf{Parameters:}\qquad &\quad  \epsilon\in(0,\exp({\underline{\mu}-1})];\quad \delta={\epsilon}/{(2\log(m+1)\|B\|_1)};\quad \eta=1;\\
    &\quad L=\exp({\bar{\mu}+\eta})+{\|B\|_1}/{\delta};\quad \sigma=\exp({\underline{\mu}-\eta});\quad q={\sigma}/{L}.\\
  \textbf{Initialization:}\qquad   &\quad  \mu^0\in\left[\underline{\mu}\1_m,\bar{\mu}\1_m\right];\quad y^0=\mu^0; \\
\textbf{Stopping Criterion:}\qquad &\quad \left\|\nabla F_{\delta}(\mu^t)\right\|_2\leq\min\left\{\sigma\epsilon,\sqrt{\sigma\epsilon} \right\};\\
\textbf{Gradient Step:}\qquad &\quad \mu^{t+1}_j =\max\left\{\underline{\mu}-\eta,\min\left\{y^t_j-\frac1L\nabla_j F_{\delta}(y^t),\bar{\mu}+\eta\right\}\right\},\ \forall\ j\in[m];\\
\textbf{Acceleration Step:}\qquad &\quad y^{t+1} = \mu^{t+1} + \frac{1-\sqrt{q}}{1+\sqrt{q}} \left(\mu^{t+1} - \mu^t\right). 
\end{empheq}

APM has an economic interpretation: It mimics a market that predicts future prices and excess supply to adjust prices. Specifically, future prices $\exp({y^t_j})$ are estimated as current prices $\exp({\mu^t_j})$ multiplied by a price momentum $\exp(\frac{1-\sqrt{q}}{1+\sqrt{q}} (\mu^{t}_j - \mu^{t-1}_j))$. Subsequently, with the smooth function $\nabla F_{\delta}$ approximating the excess supply, the future excess supply is predicted by $\nabla F_{\delta}(y^t)$. Then, instead of reacting to the current excess supply $v\in\partial F(\mu^t)$, the market uses the predicted future excess supply $\nabla F_{\delta}(y^t)$ to adjust the prices. As we shall show, through such a price-adjustment process, the prices will converge to approximate CE prices at a faster rate.

To derive the convergence rate for APM, we first define the approximate CE measure by the function value gap $F(\mu)-\inf F$, which equals zero if and only if $\mu_j=\mu^*_j=\log(p^*_j),j\in[m]$. We note that this approximate CE measure is stronger than the distance square measure in \cite[Theorem 1 and Corollary 1]{chaudhury2024competitive}, as $\|p-p^*\|_2^2\leq\cO(\epsilon)$ can be implied by $f(p)-f(p^*)=F(\mu)-F(\mu^*)\leq \cO(\epsilon)$ due to the quadratic growth of $f$ (see \citet[Lemma 4]{nan2024convergence}).
\begin{definition}[Approximate CE Prices]\label{def:approximate}
   We say that $p\in\R^m$ is an $\epsilon$-CE price vector if for $\mu_j=\log(p_j),j\in[m]$, one has $F(\mu)-\inf F\leq \epsilon$.\end{definition}
Now, we are prepared to establish the convergence rate for APM.
\begin{theorem}\label{th:SAG}
     APM finds an $\epsilon$-CE price vector in at most $\tilde{\cO}\left(\frac{1}{\sqrt{\epsilon}}\right)$ iterations.
\end{theorem}
To the best of our knowledge, APM is the first price-adjustment process to achieve an iteration complexity of $\tilde{\cO}({1}/{\sqrt{\epsilon}})$ in (quasi-)linear Fisher markets, representing a significant improvement over t\^atonnement methods. We thereby address (Q1). Furthermore, we show that the prices produced by  APM can be utilized to compute an approximate CE allocation. 
\begin{proposition}\label{pro:appriximate}
    Consider the output $\mu$ of APM that satisfies the stopping criterion with $\epsilon\leq \log(m+1)\|B\|_1$. The corresponding prices $p_j=\exp(\mu_j), j\in[m]$ and allocation $x\in\R^{n\times m}_+$ given by
    $x_{ij}=\frac{B_i}{p_j}\cdot\frac{\exp({\frac{\log(v_{ij})-\mu_j}{\delta}})}{\sum_{j\in\{0\}\cup[m]}\exp({\frac{\log(v_{ij})-\mu_j}{\delta}})}$, $ i\in[n], j\in[m],$
     satisfy the following: 
    \begin{enumerate}[label={{\rm (B\arabic*).}},itemsep=0.5pt,leftmargin=*]
    \item $\sum_{j\in[m]}p_jx_{ij} \leq  B_i$ for all $i\in[n]$;
    \item $u_i\left(x_i\right)+\alpha^{-1}B_i \geq \left(1-\frac{2\epsilon}{\|B\|_1}\right)\underset{x^{\prime}_i \in \R_{+}^{n}} \max \left\{u_{i}(x^{\prime}_i)+\alpha^{-1}B_i:\sum_{j\in[m]} p_jx^{\prime}_{ij} \leq  B_i \right\}$ for all $i\in[n]$;
    \item $ |\sum_{i\in[n]}x_{ij}-1|\leq \epsilon$ for all $j\in[m]$.
\end{enumerate}
\end{proposition}
Note that (B1) coincides with (A1), while (B2)\footnote{(B2) includes $\alpha^{-1}B_i$ to ensure that both sides remain nonnegative under quasi-linear utilities, where $\alpha=1$.} and (B3) are approximate versions of (A2) and (A3), respectively. We see that the computed allocation $x$ and price $p$ satisfy the CE conditions up to an error of $\cO(\epsilon)$. Therefore, APM yields not only approximate CE prices but also an approximate CE allocation. This additional advantage stems from our smoothing strategy and stopping criterion based on $\|\nabla F_{\delta}(\mu^t)\|$. In comparison, traditional t\^atonnement methods, due to their nonsmooth formulations (typically \eqref{eq:unified-Dual}), lack such a stopping criterion, making it difficult for them to compute an approximate CE allocation. 

\section{Recovery of Exact CE} \label{sec:exact-CE}
Though APM finds an approximate CE price at a fast rate, it shares the same drawback with existing price-adjustment methods, i.e., the lack of a convergence guarantee to exact CE prices. To address this gap, we explore the relationship between approximate and exact CE prices. It turns out that through the optimality conditions of our formulation \eqref{eq:QL-exp} at the optimal solution $\mu^*$, one can develop a  \textit{recovery oracle} $\cR$ to bridge them, as specified in the following theorem. 
\begin{theorem}[Recovery of Exact CE]\label{th:recovery}
There exists a constant $\Delta^*>0$ and an oracle $\cR$ that maps $\mu\in\R^m$ and $r\in\R$ to $\mu^*$ (denoted by $\cR(\mu,r)=\mu^*$) whenever $\mu\in\B(\mu^*,r)$ and $0<r<\Delta^*/4$. Furthermore, the oracle $\cR$ completes this mapping in at most $\cO((m+n)^2)$ arithmetic operations.
\end{theorem}
Theorem \ref{th:recovery} is significant as it not only reveals that the exact CE prices can be recovered from its nearby points, i.e., $\epsilon$-CE prices, but also demonstrates the low computational cost of the recovery oracle $\cR$, which is comparable to a single t\^atonnement iteration. With the recovery oracle in hand,
one can refine APM or other price-adjustment methods by coupling them with $\cR$, thereby achieving an iterate convergence guarantee to CE, as demonstrated in Sec. \ref{subsec:adaptive}.

The construction of the recovery oracle $\cR$ is rooted in the local property of our formulation \eqref{eq:QL-exp} and the so-called \textit{connection class} (see Definition \ref{def:connection}) of the active index sets
\[\cJ_i(\mu)\coloneqq \argmax_{j\in\{0\}\cup[m]}\left\{\log(v_{ij})-\mu_j\right\},\quad\ i\in[n].\] 
The key observations are: (i) given all the connection classes of $\cJ_i(\mu^*)$, $i\in[n]$, the optimal solution $\mu^*$ can be exactly calculated via some basic mathematical operations; (ii) there exists a radius $\Delta^*>0$ such that $\cJ_i(\mu^*)$, $i\in[n]$, along with their connection classes, can be computed through an arbitrary point $\mu\in\R^m$ with $\|\mu-\mu^*\|_2<\Delta^*/4$. See Appendix \ref{subsec:recovery} for construction details.

For the radius $\Delta^*$, we remark that it only depends on the optimal solution $\mu^*$ and the utility parameters $v_{ij}$, which ultimately reduces to dependence only on the parameters $v_{ij}$, $B_i$. Indeed,  $\Delta^*$ is given by $\Delta^*\coloneqq\Delta(\mu^*)$, where the function $\Delta:\R^m\to\R$ is defined as follows:
\begin{equation} 
\Delta_{i}(\mu) \coloneqq \max_{j\in\{0\}\cup[m]}\{\log(v_{ij})-\mu_j\} - \max_{j\in\left(\{0\}\cup[m]\right)\setminus \cJ_i(\mu)}\{\log(v_{ij})-\mu_j\}; \quad \Delta(\mu)\coloneqq \min_{i\in[n]}\ \Delta_{i}(\mu).   \label{eq:def-delta}\end{equation}
The above definition yields a natural economic interpretation. Specifically, write  
$\mu^*_j=\log(p^*_j)$ and $\Delta_i(\mu^*)=\log(\frac{v_{ij_1}}{p^*_{j_1}}/\frac{v_{ij_2}}{p^*_{j_2}})$, where $j_1\in\cJ_i(\mu^*)$ and $j_2\in \argmax_{j\in(\{0\}\cup[m])\setminus\cJ_i(\mu^*)}\{\log(v_{ij}/p^*_j)\}$. 
We see that $\Delta_i(\mu^*)$ represents the logarithm of the ratio between the highest and the second highest bang-per-buck value for buyer $i$. Hence,
$\Delta^*=\min_{i\in[n]}\Delta_i(\mu^*)$ indicates the smallest bang-per-buck gap between buyers' top two choice at CE. Intuitively, a larger gap makes it easier to rank demands and determine CE prices across goods, which aligns with the result of Theorem \ref{th:recovery}.
\begin{remark}\label{re:recovery}
With $\mu^*$ obtained, the CE prices can be recovered by letting $p_j^*=\exp({\mu^*_j})$, $j\in[m]$. Furthermore, one can recover the CE allocation $x^*$ by computing a feasible solution $\lambda^*\in\R^{n\times(m+1)}$ for the following linear system  and letting $x^*_{ij}=\lambda^*_{ij}/p^*_j$, $i\in[n],j\in[m]$.
\begin{equation}\label{eq:test_stationarity}
      \begin{array}{rlrll}
        \sum\limits_{j\in\{0\}\cup[m]}\lambda_{ij}&=B_i,\qquad&\forall\  i\in[n]; \qquad    \sum\limits_{i\in [n]}\lambda_{ij}&=\exp({\mu^*_j}),\qquad &\forall\ j\in[m];  \\
      \lambda_{ij}&\geq0,\qquad &\forall\ j\in \{0\}\cup[m],i\in[n];  \qquad\lambda_{ij}&=0,\qquad &\forall\ j\notin \cJ_i(\mu^*),i\in[n].\\
    \end{array}
\end{equation}
\end{remark}
\subsection{Construction of Recovery Oracle}\label{subsec:recovery}
The foundation of our recovery oracle lies in the following observation: 
In Problem \eqref{eq:QL-exp}, if the active index sets at the optimal solution $\mu^*$ are obtained, then $\mu^*$ can be computed directly from the optimality conditions. To illustrate this, we recall the maximum functions $h_i(\mu)= \max_{j\in\{0\}\cup[m]}\left\{\log(v_{ij})-\mu_j\right\}$ and the index functions $\cJ_i(\mu)= \argmax_{j\in\{0\}\cup[m]}\left\{\log(v_{ij})-\mu_j\right\}$, $i\in[n]$.
We have the following optimality condition of \eqref{eq:QL-exp}:
\begin{equation}
\begin{aligned}
\begin{pmatrix}
     \exp({\mu_1}),&\exp({\mu_2}),&\ldots,&\exp({\mu_m})
\end{pmatrix}-\sum_{i\in[n]}\lambda_i=0;\\
\text{where }\quad\lambda_i\in-B_i\partial h_i(\mu)=\conv\left\{B_i\e_j:j\in\cJ_i(\mu)\right\}.
\end{aligned} \label{eq:optimality2}   
\end{equation}
The key to solving this equation is thoroughly examining the intersection properties of active index sets $\cJ_i(\mu),i\in[n]$. For this purpose, we define the so-called connection class for $\cJ_i(\mu),i\in[n]$.
\begin{definition}[Connection Class of Index Sets]\label{def:connection}
    Given nonempty index sets $J_i\subseteq[n],i\in[n]$, we say that $J_i$ connects $J_{i^{\prime}}$, denoted by $J_i\sim J_{i^{\prime}}$, if $J_i\cap J_{i^{\prime}}\neq\emptyset$ or there exist indices $i_1,i_2,\ldots,i_s,s\in[n]$ such that 
    \[J_i\cap J_{i_1}\neq\emptyset;\quad J_{i_1}\cap J_{i_2}\neq\emptyset;\quad\cdots\quad;J_{i_{s-1}}\cap J_{i_s}\neq\emptyset;\quad  J_{i_s}\cap J_{i^{\prime}}\neq\emptyset. \]
    Otherwise, we say that $J_i$ disconnects $J_{i^{\prime}}$, denoted by $J_i\nsim J_{i^{\prime}}$.
   Further, we say that $\{J_i:i\in I\subseteq[n]\}$ form a connection class if $J_i\sim J_{i^{\prime}}$ for all $i,i^{\prime}\in I$ and $J_i\nsim J_{i^{\prime}}$ for $i\in I,i^{\prime}\in[n]\setminus I$.
\end{definition}
Clearly, the connection class is an equivalence class and satisfies (i) $J_{i}\sim J_{i}$; (ii) if $J_i\sim J_{i^{\prime}}$, then $ J_{i^{\prime}}\sim J_i$; (iii) if $J_{i_1}\sim J_{i_2}$ and $J_{i_2}\sim J_{i_3}$, then $J_{i_1}\sim J_{i_3}$. Moreover, each index set $J_i$ is contained in a unique connection class. Based on the connection class, we have the following lemma that serves as the foundation of the recovery oracle.
\begin{lemma}\label{le:recover}
    Let $\mu^*$ be the optimal solution of Problem \eqref{eq:QL-exp}. Suppose that the index sets $\cJ_i(\mu^*)=J_i^*,i\in[n]$ are given. Let $\{J^*_i:i\in I^*_l\},l\in[s]$ be all connection classes of $J_i^*,i\in[n]$ and $\tilde{J}_l^*=\bigcup_{i\in I^*_l}J^*_i,l\in[s]$. Then, $[m]\subseteq\bigcup_{l\in[s]}\tilde{J}_l^*=\bigcup_{i\in[n]}J^*_i$ and the optimality condition \eqref{eq:optimality2} reduces to the following equations over $l\in[s]$, which admits an exact solution:
    \begin{gather}
    \label{eq:mu*1}    \log(v_{ij})-\mu_j=\log(v_{ij^{\prime}})-\mu_{j^{\prime}},\quad\forall\ j,j^{\prime}\in J_i^*,\ i\in I^*_l;\\
    \label{eq:mu*2}  \sum_{j\in\tilde{J}^*_l}\exp(\mu_j)= \sum_{i\in I^*_l}B_i\quad\text{if}\quad0\notin \tilde{J}^*_l. \qquad\qquad\quad
    \end{gather}
\end{lemma}
Given Lemma \ref{le:recover}, two natural questions arise: (i) how to derive the active index sets $\cJ_i(\mu^*),i\in[n]$; (ii) how to classify the connection classes of the active index sets efficiently. For question (i), we show that $\cJ_i(\mu^*),i\in[n]$ can be deduced from the relaxed active index sets at nearby points of $\mu^*$.
\begin{lemma}\label{le:index}
          Let $\mu^*$ be the optimal solution of Problem \eqref{eq:QL-exp} and $\mu\in\B(\mu^*,r)$ with $0<r<\Delta^*/4$. Then, we have
 $\cJ_i(\mu^*)=\{j\in\{0\}\cup[m]:\log(v_{ij})-\mu_j\geq h_i(\mu)-2r\}$, $i\in[n]$.
    \end{lemma}

    \begin{algorithm}[htbp]
	\caption{Classification Procedure $\cE$}
	\begin{algorithmic}[1]
		\REQUIRE{Index sets $J_i\subseteq\{0\}\cup[m]$, $i\in[n]$ }
 \STATE $s=0$, $I^c=[n]$ \COMMENT{$I^c$ represents the set of unclassified indices }
 \WHILE{$I^c\neq\emptyset$}
 \STATE $s=s+1$,  
 \STATE Select an arbitrary index $i^s\in I^c$ \COMMENT{Find an unclassified index $i^s\in[n]$}
 \STATE $I_s=\{i^s\}$, $I^c=I^c\setminus\{i^s\}$
 \STATE   $\tilde{J}_s=\bigcup_{i\in I_s}J_i=J_{i^s}$, $J_{check}=\emptyset,t=1$
 \WHILE{$J_{check}\subsetneqq\tilde{J}_s$}
 \STATE Select an arbitrary index $j^s_t\in \tilde{J}_s\setminus J_{check}$\COMMENT{Use indices of $\tilde{J}_s$ to find connected sets}
 \STATE $\cI_{new}(j^s_t)=\{i\in I^c:j^s_t\in J_i\}$ 
 \STATE $I_s=I_s\cup \cI_{new}(j^s_t)$
  \STATE  $I^c=I^c\setminus\cI_{new}(j^s_t) $
 \STATE $\cJ_{new}(j^s_t)= \bigcup_{i\in \cI_{new}(j^s_t)}J_i$;  
  \STATE $\tilde{J}_s=\tilde{J}_s\cup \cJ_{new}(j^s_t)$
 \STATE $J_{check}=J_{check}\cup\{j^s_t\}$, $t=t+1$ \COMMENT{We see that $j^s_t\in J_{i^s}\cup(\bigcup_{\tau\in[t-1]}\bigcup_{i\in\cI_{new}(j^s_{\tau})}J_{i})$}
 \ENDWHILE
 \ENDWHILE
        \ENSURE{$I_l,\tilde{J}_l$ for $l\in[s]$; $i^l,j^l_t,\cI_{new}(j^l_t)$ for $t\in[|\tilde{J}_l|]$, $l\in[s]$}
	\end{algorithmic}
	\label{al:class} 
\end{algorithm}
    
For question (ii), we develop an algorithm, i.e., Algorithm \ref{al:class}, that classifies the connection classes of the input index sets $J_i\subseteq\{0\}\cup[m],i\in[n]$. The main idea is to iteratively incorporate sets that share the elements of the current connected sets. Specifically, we consider the connected sets $J_{i},i\in I$ and select an index $j$ from their combination, i.e., $j\in \tilde{J}\coloneqq\bigcup_{i\in I}J_{i}$. By identifying the sets $J_i,i\in\cI_{new}(j)$ that also contain the element $j$, we enlarge the connected sets to $J_{i},i\in I\cup\cI_{new}(j)$, and then update $\tilde{J}=\bigcup_{i\in I\cup\cI_{new}}J_{i}$, $I=I\cup\cI_{new}$. Repeat this enlarging process until no new connected sets can be found. Then, the connected sets obtained form a connection class. To identify other connection classes and complete the classification, we simply start with unclassified sets and repeat the enlarging process. The validity of Algorithm \ref{al:class} is ensured by the following lemma. 
    \begin{lemma}\label{le:class}
    The classification procedure $\cE$ (i.e., Algorithm \ref{al:class}) satisfies the followings:
    \begin{enumerate}[{\rm (i)}]
        \item The output $\{J_i:i\in I_l\}$, $l\in[s]$ form connection classes of $J_i,i\in[n]$ with $\bigcup_{i\in I_l}J_i=\tilde{J}_l$.
        \item The computational cost of $\cE$ is at most $\cO((m+n)^2)$.
    \end{enumerate}
\end{lemma}
In each iteration, Algorithm \ref{al:class} records the indices of the newly incorporated sets that contain the index $j^l_t$. These indices facilitate an explicit solution procedure $\cP$ (Algorithm \ref{al:solution}) for the equations \eqref{eq:mu*1} and \eqref{eq:mu*2}, allowing parallel computation across different $l$. The underlying idea of $\cP$ is to utilize the index sets $\cI_{new}(j^l_t),t\in[|\tilde{J}_l|]$ and \eqref{eq:mu*1} to compute the constants $a_j,j\in\tilde{J}_l$ such that $\mu_j=\mu_{j^l_1}+a_j$, and then substitute $\mu_j=\mu_{j^l_1}+a_j$ into  \eqref{eq:mu*2} to derive a solution.
 \begin{algorithm}[htbp]
     \caption{Solution Procedure $\cP$}
     \begin{algorithmic}[1]
          \REQUIRE{$I_l,\tilde{J}_l$ for $l\in[s]$; $i^l,j^l_t,\cI_{new}(j^l_t)$ for $t\in[|\tilde{J}_l|]$, $l\in[s]$ }
\FOR{$l=1,2,\ldots,s$}
\STATE $a_{j^l_1}=0$, $J_{done}= J_{i^l}$
\STATE $a_j=\log(v_{i^lj})-\log(v_{i^lj^l_1}),\quad\forall\ j\in J_{i^l}$ 
\COMMENT{$a_{j^l_t}$ is computed before the $t$-th inner iteration}
\FOR{$t=1,\ldots,|\tilde{J}_l |$}  
\STATE $a_{j}=a_{j^l_t}+\log(v_{ij})-\log(v_{ij^l_t}),\quad\forall\ j\in J_i\setminus J_{done}, i\in\cI_{new}(j^l_t)$ 
\STATE $J_{done}=J_{done}\cup(\bigcup_{i\in \cI_{new}(j^l_t)}J_i)$
\ENDFOR
\IF{$0\in\tilde{J}_l$}
\STATE $\mu_{j^l_1}=-a_0$\COMMENT{$a_{j},j\in \tilde{J}_l$ are computed in the "for" loop w.r.t. $t$}
\STATE $\mu_j=\mu_{j^l_1}+a_j,$ \quad $\forall\ j\in\tilde{J}_l$ 
\ELSE
\STATE $\mu_{j^l_1}=\log\left(\sum_{i\in I_l}B_i\right)-\log\left(\sum_{j\in \tilde{J}_l}\exp({a_j})\right)$
\STATE $\mu_j=\mu_{j^l_1}+a_j,$ \quad $\forall\ j\in\tilde{J}_l$
\ENDIF
\ENDFOR
     \ENSURE{$\mu$}
     \end{algorithmic}
    \label{al:solution}
 \end{algorithm}
 \begin{lemma}[Property of $\cP$]\label{le:solution}
     Let $\mu^*$ be the optimal solution of Problem \eqref{eq:QL-exp} and $J^*_i=\cJ_i(\mu^*)$, $i\in[n]$. We have $\cP(\cE(\{J^*_i,i\in[n]\}))=\mu^*$, where the computational cost of $\cP$ is at most $\cO(mn)$. 
 \end{lemma}
Now, we are ready to give our recovery oracle $\cR$ in Algorithm \ref{al:guiding}. With the input $\mu\in\R^m$ and $r\in\R$, we first define $J_i\coloneqq\{j\in\{0\}\cup[m]:\log(v_{ij})-\mu_j\geq h_i(\mu)-2r\}$. By Lemma \ref{le:index}, we know that $J_i=\cJ_i(\mu^*)$ when $\mu\in\B(\mu^*,r)$ for $0<r<\Delta^*/4$. Then, we apply the classification procedure $\cE$ and obtain the connection classes of $J_i,i\in[n]$.  Subsequently, the solution procedure $\cP$ is implemented to output a vector $\hat{\mu}$.  By Lemma \ref{le:solution}, the recovery $\hat{\mu}=\mu^*$ is guaranteed when $J_i=\cJ_i(\mu^*)$. Hence, we have $\cR(\mu,r)=\mu^*$ for $\mu\in\B(\mu^*,r)$ with $0<r<\Delta^*/4$. Moreover, the computational cost of $\cR$ is dominated by the classification procedure $\cE$ and solution procedure $\cP$, leading to a total cost of $\cO((m+n)^2)$ by Lemma \ref{le:class} and \ref{le:solution}.  We summarize these results in the following proposition, which directly proves Theorem \ref{th:recovery}.
\begin{algorithm}[htbp]
	\caption{Recovery oracle $\cR$}
	\begin{algorithmic}[1]
		\REQUIRE{Parameter $r>0$ and vector $\mu\in\R^m$   }
  \STATE Let $h_i= \max_{j\in\{0\}\cup[m]}\{\log(v_{ij})-\mu_j\}$, 
  $i\in[n]$  
    \STATE Compute the index sets $J_i=\{j\in\{0\}\cup[m]:\log(v_{ij})-\mu_j\geq h_i-2r\}$, $i\in[n]$
    \IF{$[m]\subseteq \bigcup_{i\in[n]}J_i$}
     \STATE $\hat{\mu}=\cP(\cE(\{J_i,i\in[n]\}))\in\R^m$
    \ELSE
    \STATE $\hat{\mu}=\text{NaN}\notin\R^m$
    \ENDIF
        \ENSURE{$\hat{\mu}$}
	\end{algorithmic}
	\label{al:guiding} 
\end{algorithm}

\begin{proposition}\label{pro:exact}
 Let $\mu^*$ be the optimal solution of Problem \eqref{eq:QL-exp} and $\Delta^*$ be defined by \eqref{eq:def-delta}. 
Given the input $\mu\in\R^m,r>0$, let $\hat{\mu}\in\R^m$ be the output of the recovery oracle $\cR$, i.e.,  $\hat{\mu}=\cR(\mu,r)$. The following holds:
\begin{enumerate}[{\rm (i)}]
  \item If the inputs $(\mu,r)$ satisfy $0<r<\Delta^*/4$ and $\mu\in\B(\mu^*,r)$, then we have $\hat{\mu}=\mu^*$.
    \item The computational cost of $\cR$ is at most $\cO((m+n)^2)$.
\end{enumerate}
\end{proposition}
\begin{proof}
(i) Given $\mu\in\B(\mu^*,r)$ and $0<r<\Delta^*/4$, Lemma \ref{le:index} ensures  $J_i=\cJ_i(\mu^*)$, $i\in[n]$. Then, we have $[m]\subseteq \bigcup_{i\in[n]}J_i$ by Lemma \ref{le:recover}, and further $\hat{\mu}=\cP(\cE(\{\cJ_i(\mu^*),i\in[n]\}))$ by Line 4 of Algorithm \ref{al:guiding}. The desired $\hat{\mu}=\mu^*$ follows from Lemma \ref{le:solution}.

(ii) Clearly, in Algorithm \ref{al:guiding}, the total computational cost of evaluating $h_i, J_i$, $i\in[n]$, and $\bigcup_{i\in[n]}J_i$ is of the order $\cO(mn)$. Hence, it suffices to show that the procedures $\cE,\cP$ have a complexity of $\cO((m+n)^2)$, which is ensured by Lemma \ref{le:class} and \ref{le:solution}.
\end{proof}
\subsection{Adaptive APM}\label{subsec:adaptive}
Given Theorem \ref{th:recovery}, a naive approach to compute the exact CE prices is to apply the oracle $\cR$ to the last iteration of APM: By the $\sigma$-strong convexity of $F$ on the box $[(\underline{\mu}-\eta)\1,(\underline{\mu}+\eta)\1]$, the output of APM with $\epsilon<\sigma\cdot(\Delta^*)^2/32$, denoted by $\mu$, satisfies $\|\mu-\mu^*\|_2<\Delta^*/4$. Hence, one may expect to recover $\mu^*$ via $\cR(\mu,\sqrt{2\epsilon/\sigma})$. However, such a scheme is impractical, as the radius $\Delta^*$ is unknown, making it impossible to set an appropriate accuracy $\epsilon$ for APM.

An improved approach is to find a sequence of APM outputs $\{\mu^k\}_{k\geq0}$ with decaying accuracy $\{\epsilon_k\}_{k\geq0}$, i.e., $F(\mu^k)-F(\mu^*)\leq\epsilon_k$ with $\epsilon_k\to0$, and subsequently compute $\hat{\mu}^k=\cR(\mu^k,\sqrt{2\epsilon_k/\sigma})$, $k\geq0$. We call this algorithm \textit{adaptive APM}. Clearly, there exists an index $K>0$ such that $\sqrt{2\epsilon_k/\sigma}<\Delta^*/4$  for all $k\geq K$. It follows that $\hat{\mu}^k=\mu^*$ for all $k\geq K$ by Theorem \ref{th:recovery}. Therefore, the adaptive APM finds CE in finite steps. 

To stop the adaptive APM, it suffices to test the optimality of $\hat{\mu}^k$, $k\geq0$ for \eqref{eq:QL-exp}. The test is equivalent to checking the feasibility of the linear system \eqref{eq:test_stationarity} with $\mu^*$ replaced by $\hat{\mu}^k$, which can be further formulated as a max-flow problem and efficiently solved, as detailed in Appendix \ref{appen:test}. We now formally present the adaptive APM as follows.
\begin{empheq}[box=\fbox]{align*}
   & \textbf{Adaptive APM}\quad\\
    \textbf{Input:}\qquad &\quad  \{\epsilon_k\}_{k\geq1} \text{ with }\epsilon_k= \theta^{k},\  \theta\in(0,1).\\
     \textbf{Inner Solver:}\qquad &\quad   \text{ APM  with strong convexity modulus }\sigma. \\
    \textbf{Stopping Criterion:}\qquad &\quad \hat{\mu}^k \text{ satisfies the optimality conditions of Problem \eqref{eq:QL-exp}.}  \\
\textbf{Approximate Step:}\qquad &\quad \textit{Run the inner solver to find } \mu^k\text{ satisfying }F(\mu^k)-F(\mu^*)\leq\epsilon_k. \\
\textbf{Recovery Step:}\qquad &\quad \hat{\mu}^k=\cR(\mu^k,\sqrt{2\epsilon_k/\sigma}). 
\label{alg:acc_tato}
\end{empheq}

\begin{theorem}\label{th:SAG_exact}
The adaptive APM finds CE prices in at most $K$ iterations with 
\[K=\max\left\{\left\lceil2\log_{\theta}\left(\frac{\sqrt{\sigma}\Delta^*}6\right)\right\rceil,1\right\}.\] 
\end{theorem}
To the best of our knowledge, Theorem \ref{th:SAG_exact} provides the first convergence guarantee to exact CE prices for price-adjustment processes in (quasi-)linear Fisher markets, thereby addressing (Q2). This demonstrates the power of our recovery oracle. Theorem \ref{th:SAG_exact} also has an economic interpretation: As the market adapts the prices with diminishing increments (i.e., $\epsilon_k\to0$ and the stepsize of the inner solver goes to zero), the iterate prices would well approximate the CE prices, providing sufficient information for sellers and buyers to determine satisfactory prices (i.e., CE prices).
\begin{remark}\label{re:recovery3}
    By combining Theorem \ref{th:SAG_exact} and \ref{th:SAG}, we see that the adaptive APM needs at most     
    \[ \sum_{k=1}^K\tilde{\cO}\left(\frac1{\sqrt{\epsilon_k}}\right)\leq \tilde{\cO}\left(\frac{K}{\sqrt{\epsilon_K}}\right)=\tilde{\cO}\left(\frac1{\Delta^*}\right)\]
     APM iterations to compute the exact CE. 
\end{remark}
\begin{remark}\label{re:recovery2}
  In the adaptive APM framework, if we substitute the inner solver with the additive t\^atonnement and replace the strong convexity modulus $\sigma$ with the quadratic growth modulus $\alpha$ in \cite[Lemma 4]{chaudhury2024competitive}, the result of Theorem \ref{th:SAG_exact} remains valid, with $\sigma$ replaced by $\alpha$. 
\end{remark}
\begin{proof}[Proof of Theorem \ref{th:SAG_exact}]
     We first show that the adaptive APM stops in $K$ iterations. Notice that the outputs $\mu^k$, $k\geq0$ of APM are in the box $[(\underline{\mu}-\eta)\1,(\underline{\mu}+\eta)\1]$ and $F$ is $\sigma$-strongly convex on this box. We have $F(\mu^k)-F(\mu^*)\geq\frac{\sigma}{2}\|\mu^k-\mu^*\|_2^2$ for all $k\geq0$. On the other hand, the approximation step ensures $F(\mu^k)-F(\mu^*)\leq\epsilon_k$. We have
 \begin{equation}\label{eq:strong}
     \|\mu^k-\mu^*\|_2\leq\sqrt{\frac{2}{\sigma}\epsilon_k},\qquad\forall\ k\geq0.
 \end{equation}
 By the definition of $K$, for $k\geq K$, we have $\epsilon_{k}\leq\theta^{K}\leq \sigma (\Delta^*)^2/36$, or equivalently, $\sqrt{2\epsilon_{k}/\sigma}\leq \Delta^*/(3\sqrt{2})$. This, together with \eqref{eq:strong} and Theorem \ref{th:recovery}, implies $\hat{\mu}^{k}=\mu^*$ for $k\geq K$. We complete the proof.
\end{proof}

\section{Experiments} \label{sec:experiments}
In this section, we demonstrate the competitive numerical performance of our APM in computing approximate CE prices. Moreover, we show that the adaptive APM can successfully recover exact CE prices, requiring much less time than the off-the-shelf solver. We run all experiments on a personal desktop that uses the Apple M3 Pro Chip and has 18GB RAM. The optimization solver implemented is Mosek version 10.2.3. 
\subsection{APM for Computing Approximate CE}
We evaluate the number of iterations needed by different algorithms for finding approximate CE prices. We consider both linear and quasi-linear utility and collect data for utility parameters from synthetic and real-world datasets under different instance sizes. The budgets $B_i$, $i\in[n]$ are set as $1$, following the commonly adopted equal income setting \cite{gao2020first}. We use the additive t\^atonnement method \cite{nan2024convergence}, the mirror descent method \cite{gao2020first,birnbaum2011distributed}, and PDHG \cite{liu2025pdhcg} as benchmarks, as they are representative of t\^atonnement, proportional response dynamics, and first-order methods respectively. Moreover, they share the same iteration computational cost of $\cO(mn)$ with APM. 
\begin{figure}[htbp]
    \centering
        \includegraphics[width=0.32\linewidth]{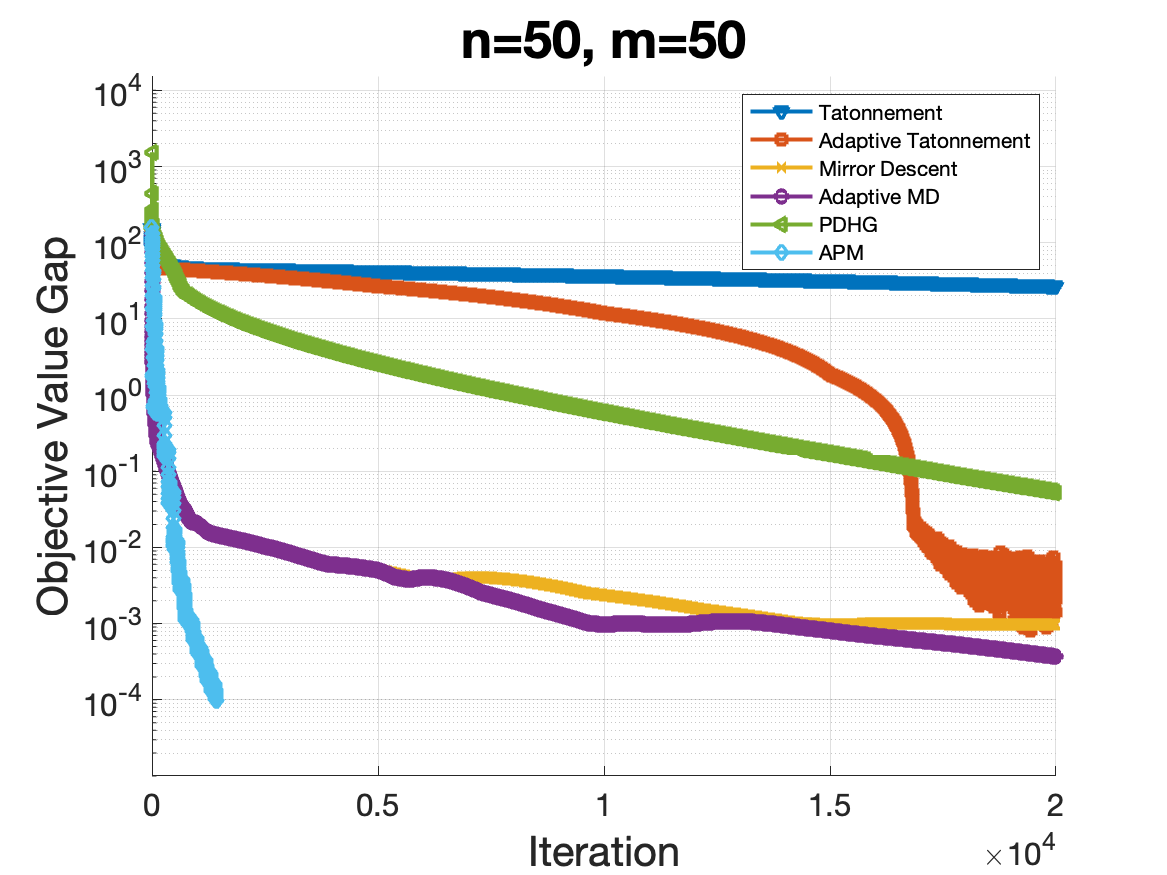} 
        \includegraphics[width=0.32\linewidth]{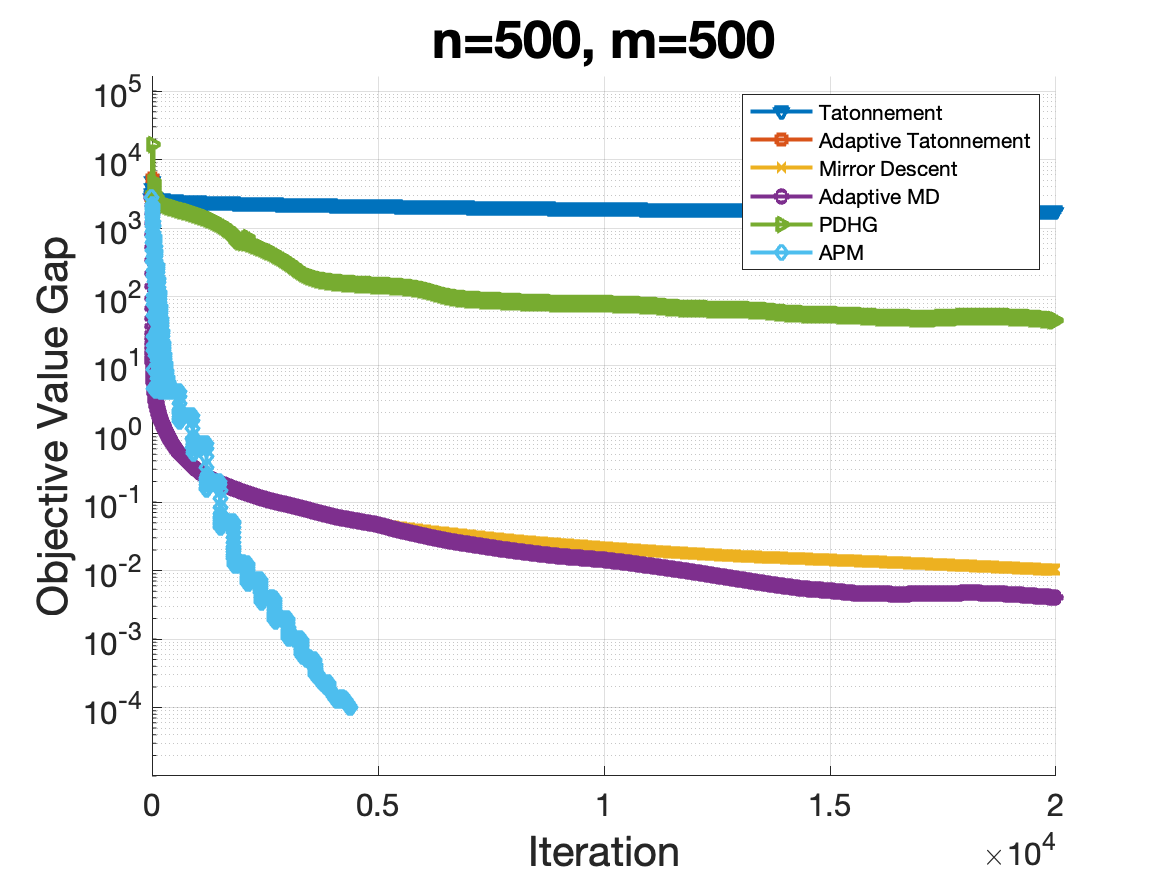} 
        \includegraphics[width=0.32\linewidth]{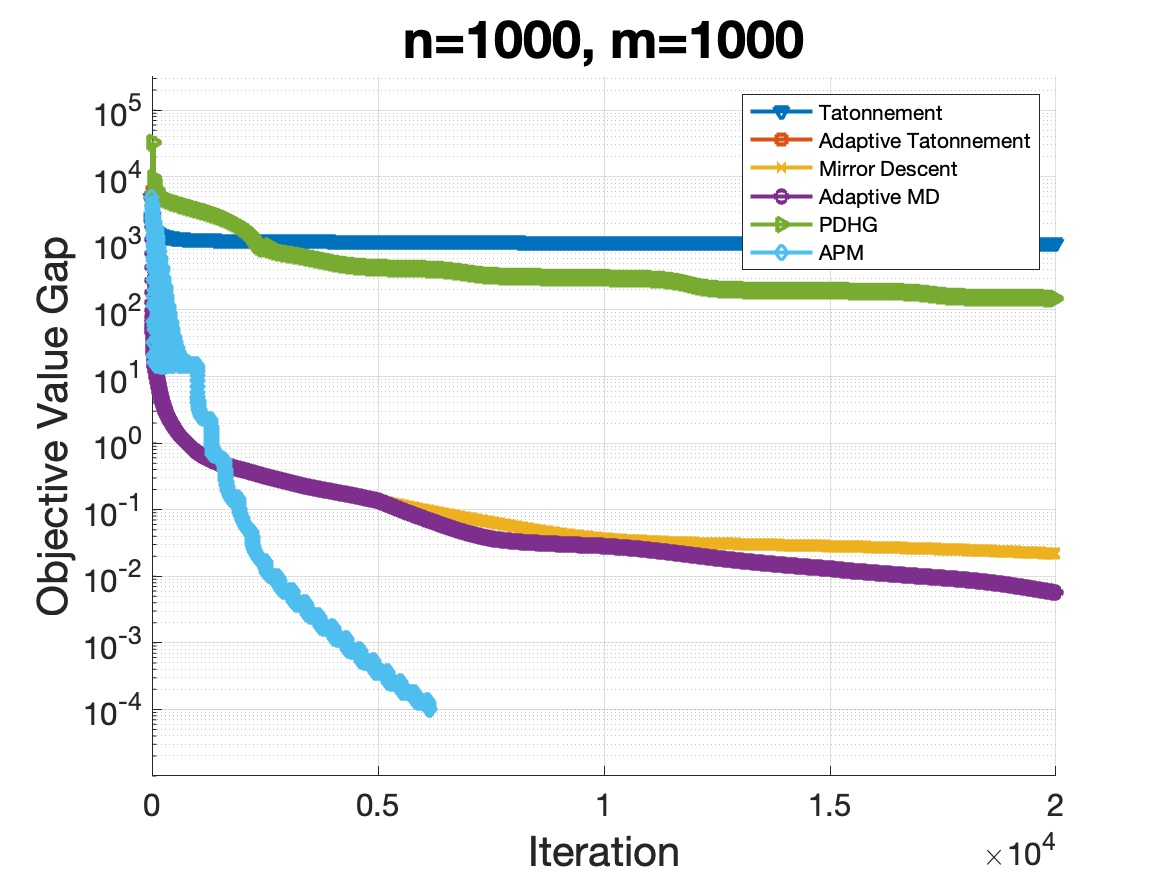} 
    \caption{Comparison of different algorithms for computing $\epsilon$-CE prices in \textit{linear} Fisher markets with uniform data, where the precision $\epsilon$ is $10^{-4}$.} 
    \label{figure:adaptive_md_linear}
\end{figure}
\begin{figure}[htbp]
	\centering
	\centering
	\includegraphics[width=0.32\linewidth]{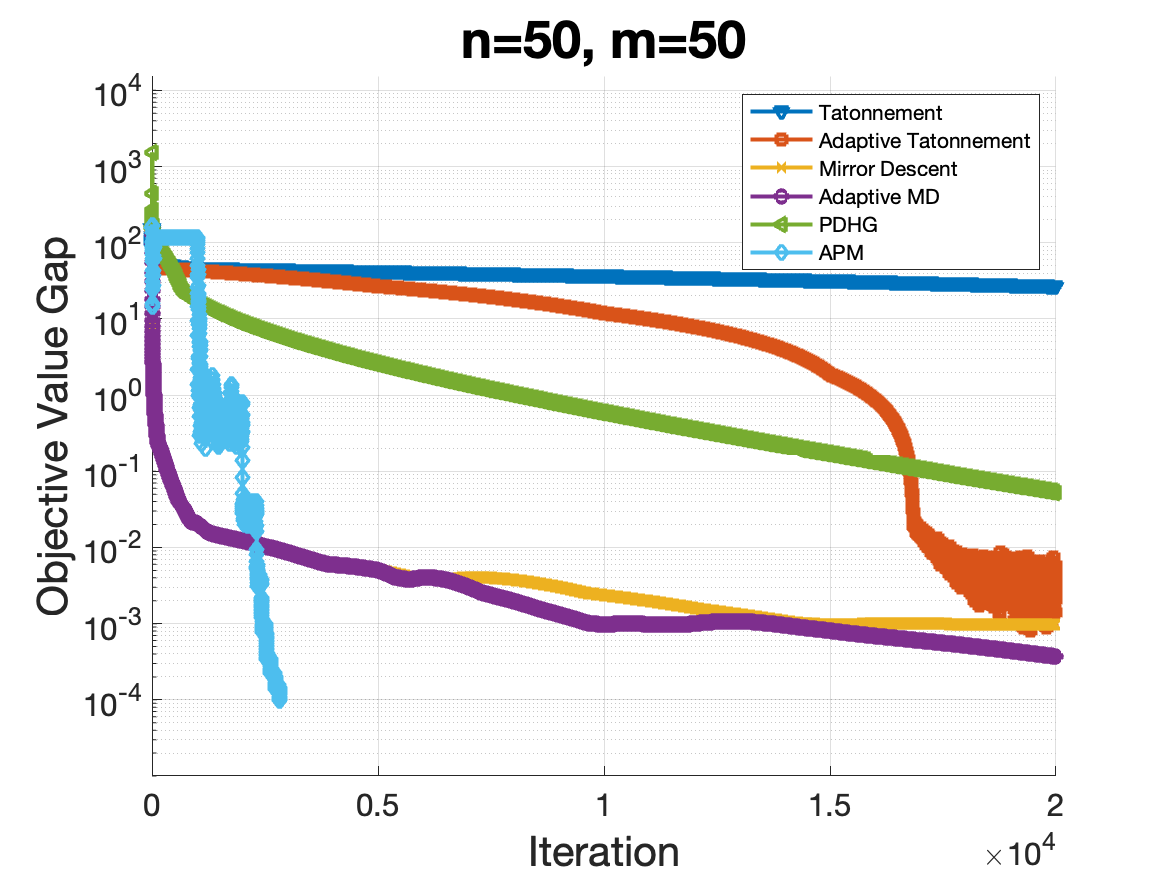} 
	\includegraphics[width=0.32\linewidth]{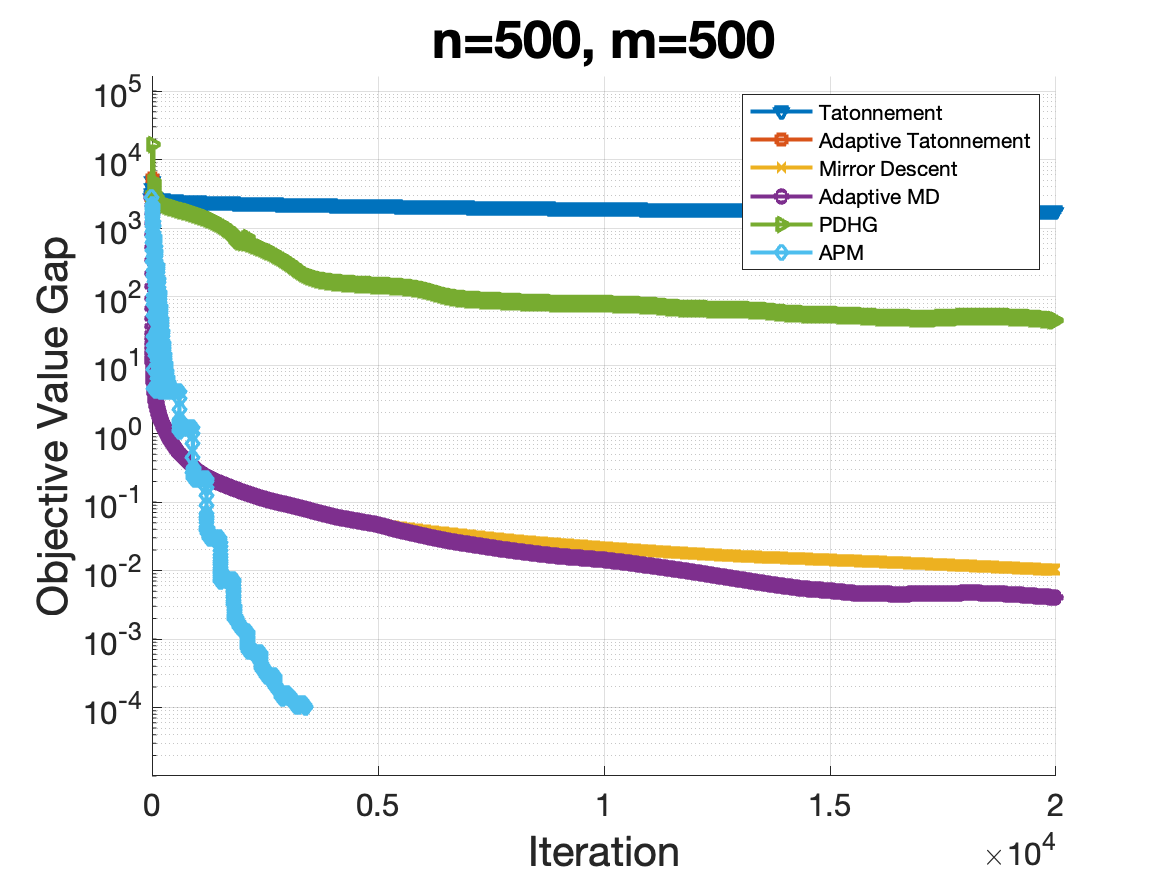} 
	\includegraphics[width=0.32\linewidth]{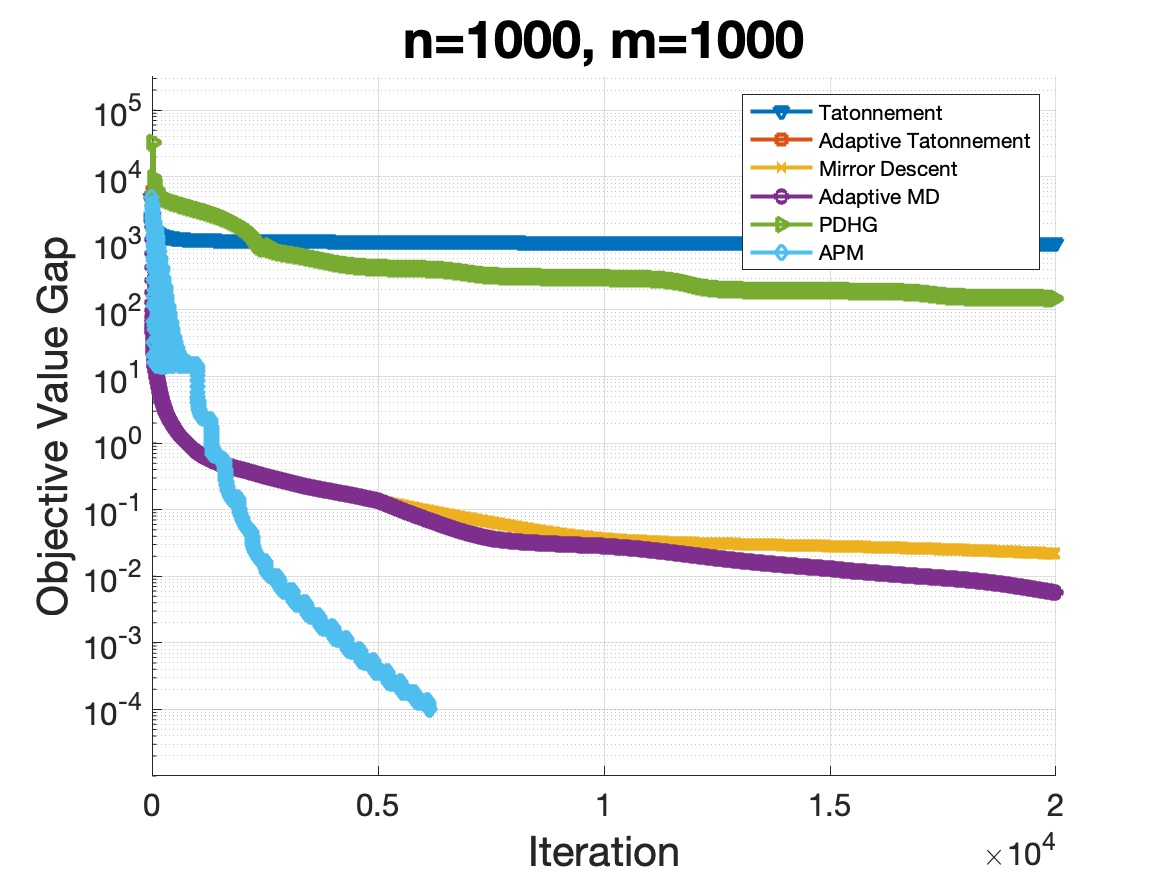} 
	\caption{Comparison of different algorithms for computing $\epsilon$-CE prices in \textit{linear} Fisher markets with exponential data, where the precision $\epsilon$ is $10^{-4}$.}
	\label{figure:exp_linear}
\end{figure}
\begin{figure}[htbp]
    \centering
        \includegraphics[width=0.32\linewidth]{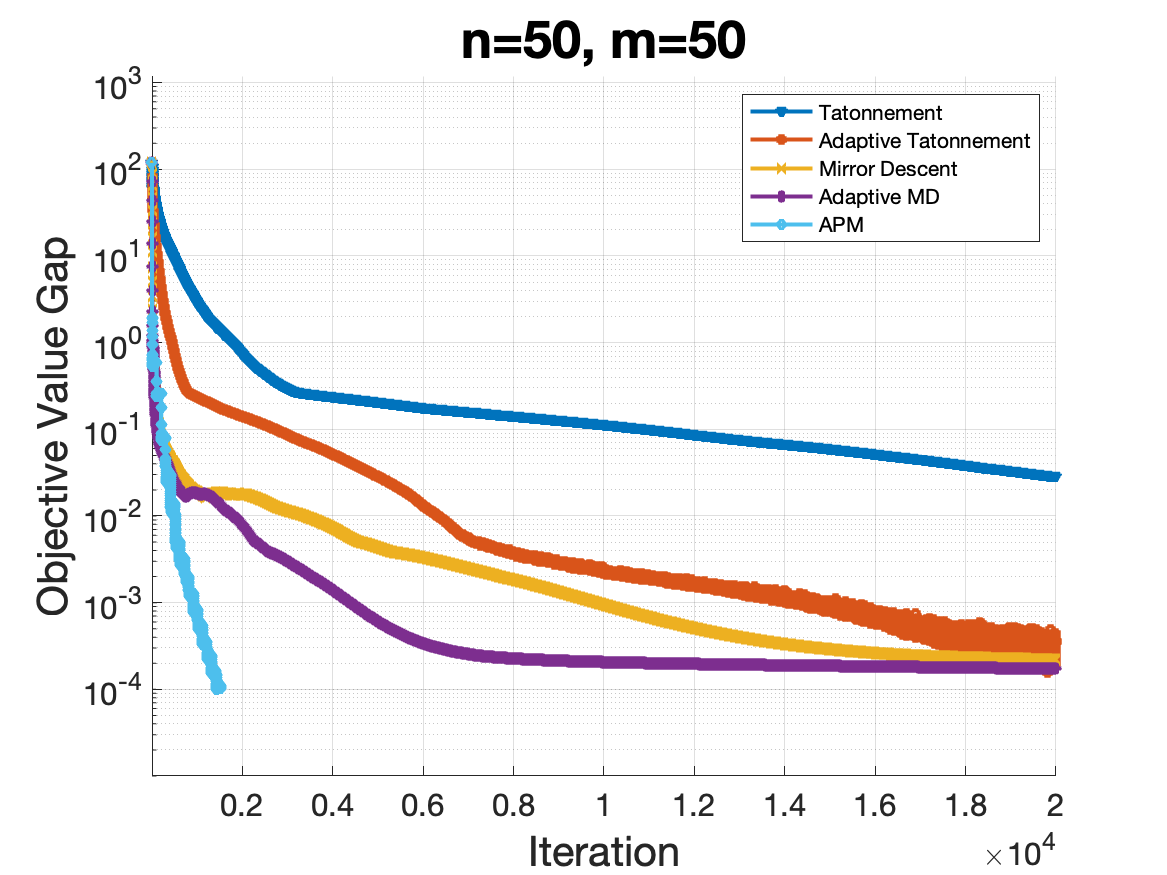} 
        \includegraphics[width=0.32\linewidth]{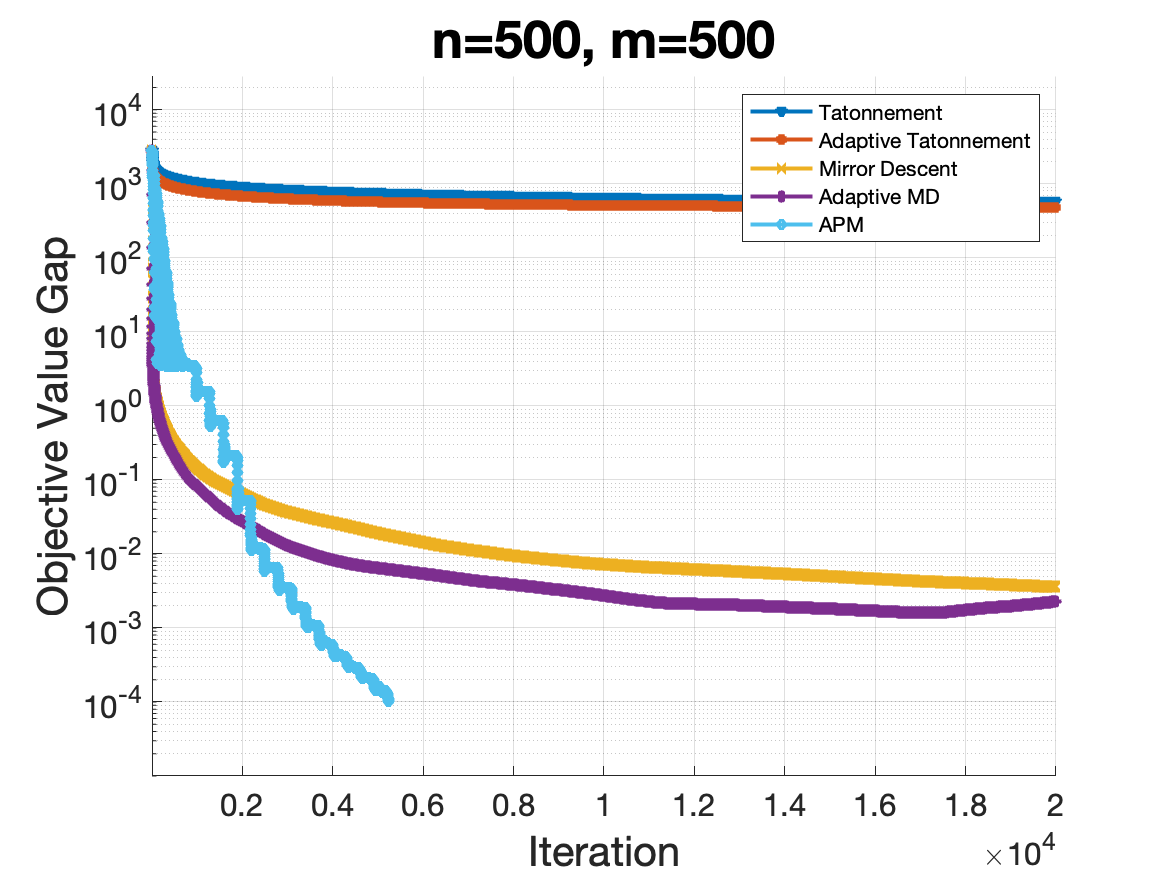} 
        \includegraphics[width=0.32\linewidth]{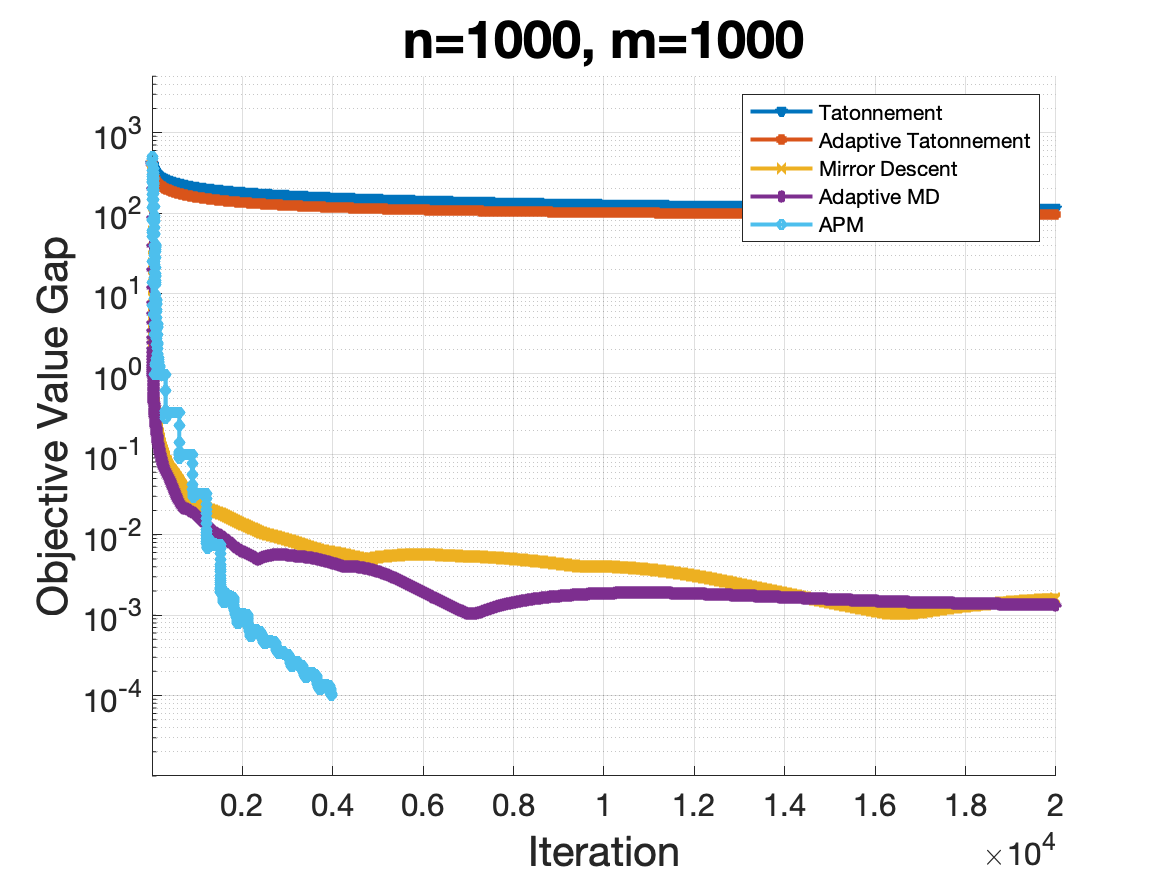} 
    \caption{Comparison of different algorithms for computing $\epsilon$-CE prices in \textit{quasi-linear} Fisher markets with uniform data, where the precision $\epsilon$ is $10^{-4}$.} 
    \label{figure:adaptive_md_quasi_linear}
\end{figure}
\begin{figure}[htbp]
	\centering
	\centering
	\includegraphics[width=0.32\linewidth]{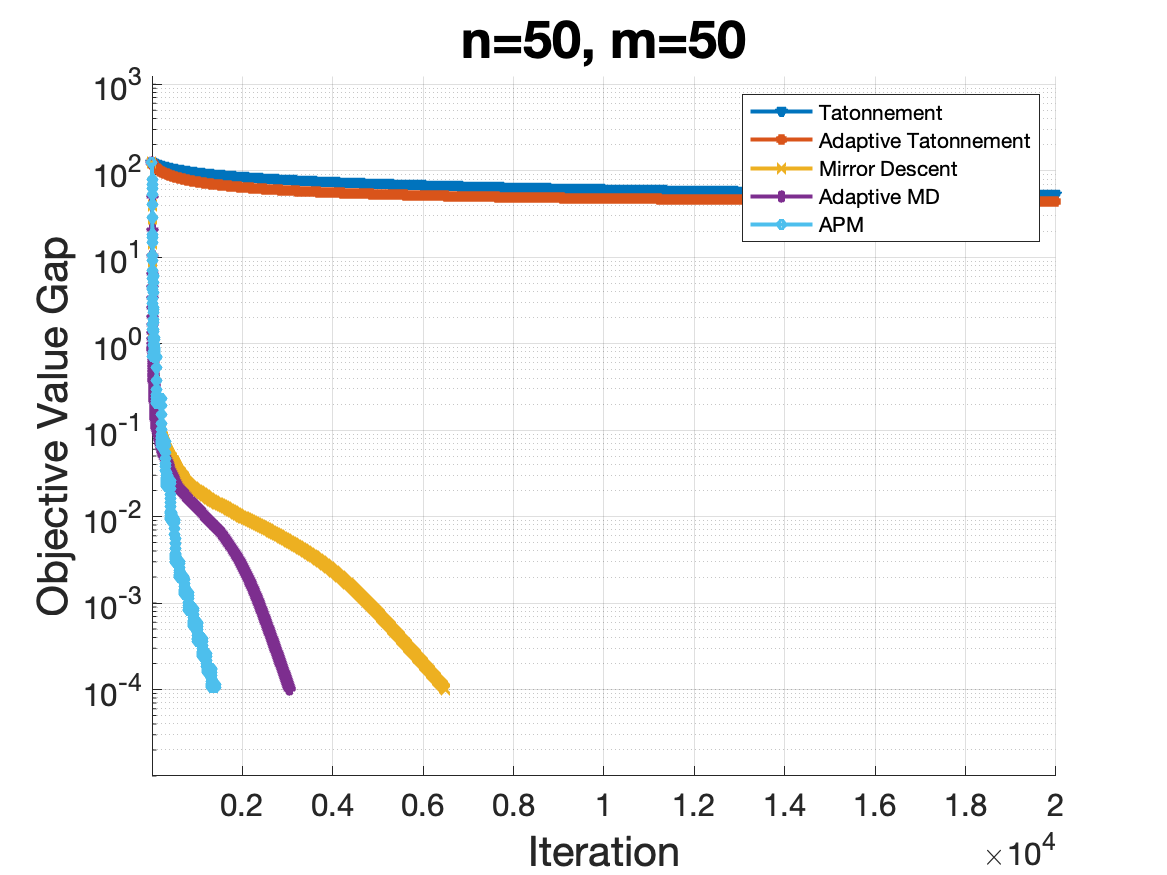} 
	\includegraphics[width=0.32\linewidth]{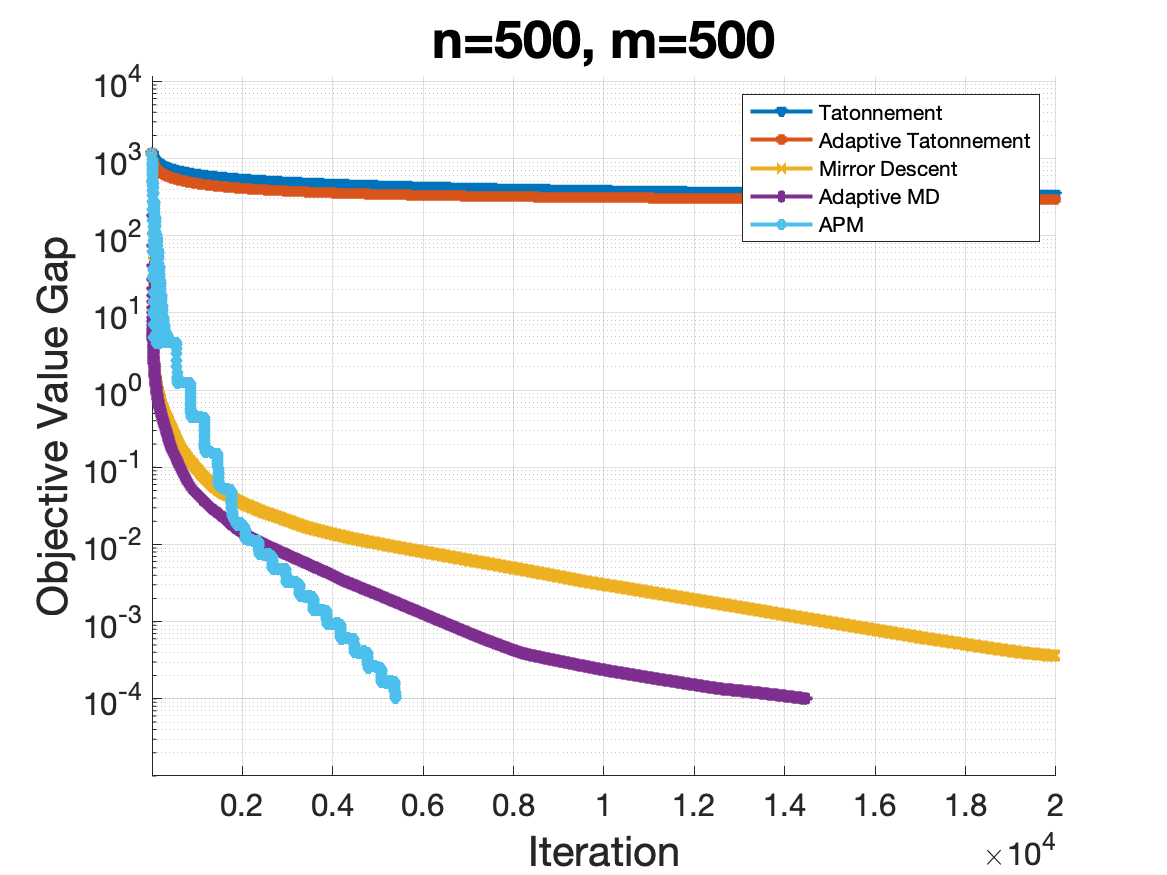} 
	\includegraphics[width=0.32\linewidth]{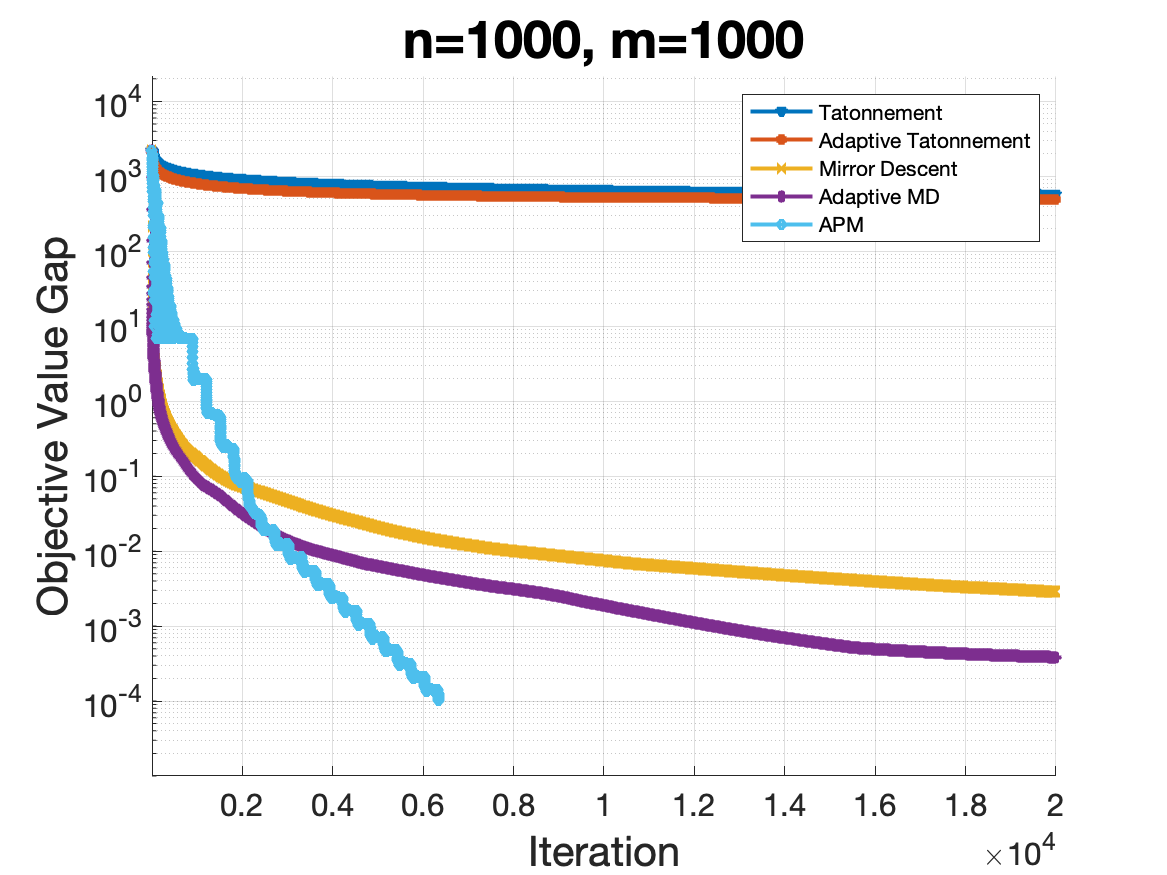} 
	\caption{Comparison of different algorithms for computing $\epsilon$-CE prices in \textit{quasi-linear} Fisher markets with exponential data, where the precision $\epsilon$ is $10^{-4}$.}
	\label{figure:exp_quasi}
\end{figure}

\begin{figure}[htbp]
\centering
 \includegraphics[width=0.35\linewidth]{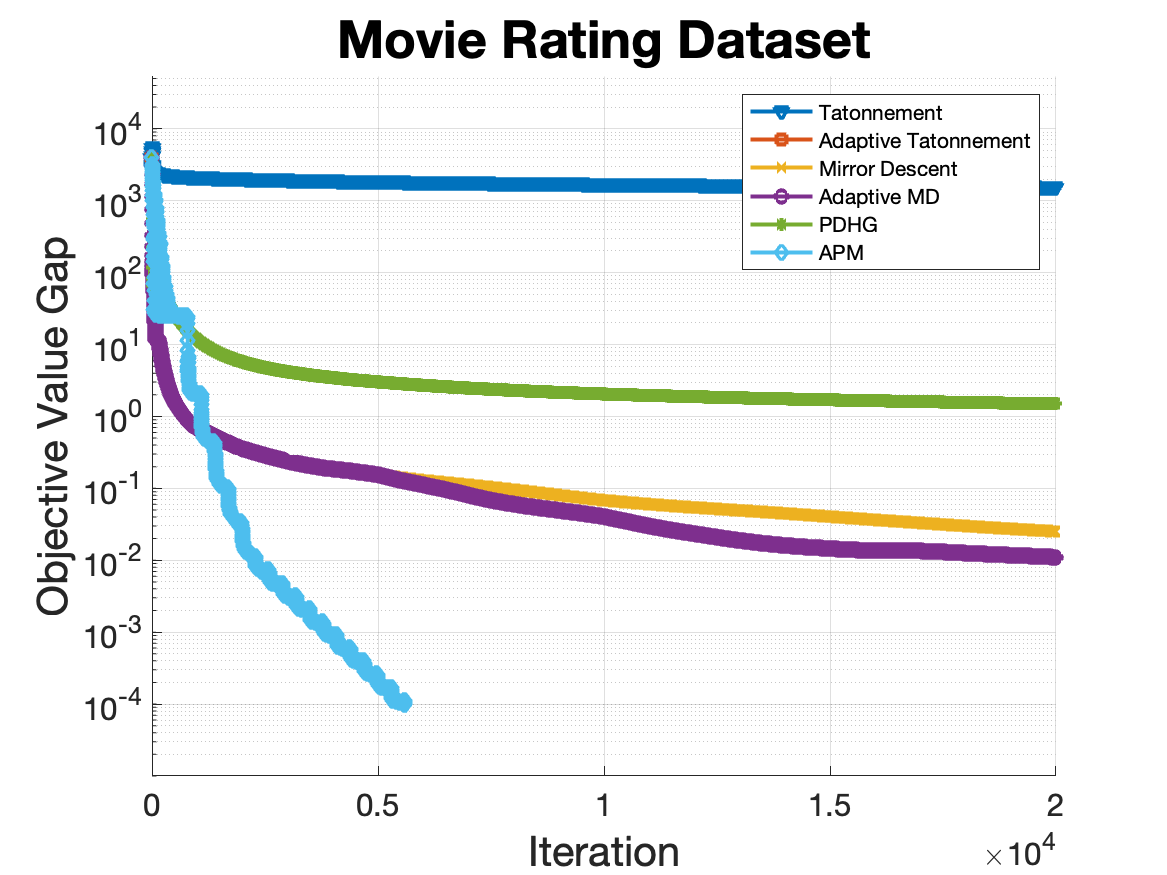} 
  \includegraphics[width=0.35\linewidth]{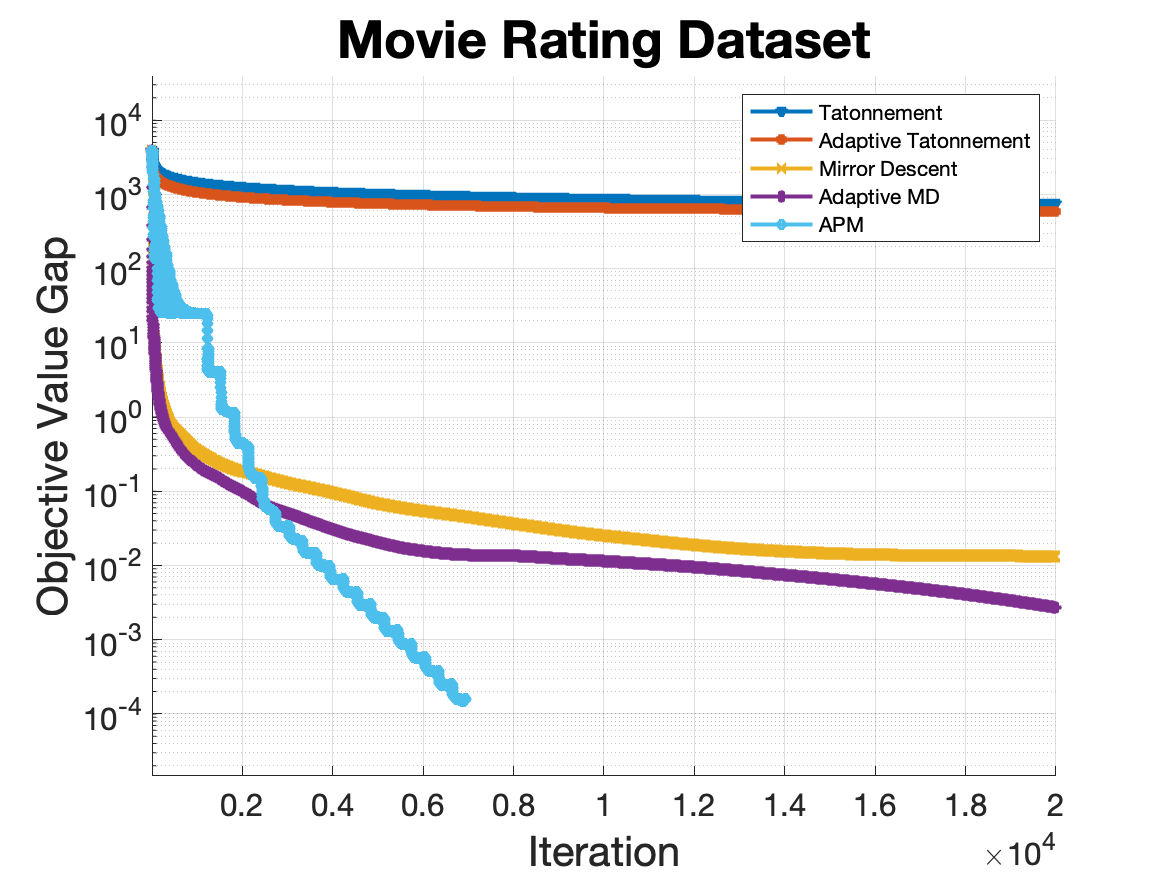} 
   \caption{Comparison of different algorithms for computing $\epsilon$-CE prices in \textit{linear} (left) and \textit{quasi-linear} (right) Fisher markets with movie rating dataset, where the precision $\epsilon$ is $10^{-4}$.} \label{figure:movie}
\end{figure}

To implement the mirror descent method to compute approximate CE prices, we use its iterates, i.e., the bid vectors, to recover prices based on their relations (see, e.g., \citep[Equation (2)]{birnbaum2011distributed}). Following \citet{gao2020first}, we set the stepsize of the mirror descent to $1$. The stepsize of additive t\^atonnement is set to $10^{-4}$. Adaptive stepsize strategy is also tested for them. Following \citet{huang2024restarted}, we use adaptive stepsize for PDHG based on the primal and dual residual. We set the stepsize of APM to $1/L$ and iteratively decrease $\delta$, where $L$ is the smoothness modulus of $F_{\delta}$. For synthetic data, the utility parameters $v_{ij}$, $i\in[n],j\in[m]$ are generated by distributions usually adopted in the literature. For real-world data, the utility matrix is constructed from a movie rating dataset \citep{Dooms2013MovieTweetings}, where each movie is treated as an item and each rater is treated as a buyer. The valuation of a buyer for an item is equal to the rating he gives to the movie. The utility matrix is completed using the matrix completion software \textit{fancyimpute}, which aligns with the data preprocessing approach of \citet{nan2024convergence}. The resulting instance has $n = 691$ buyers and $m = 632$ items.

The numerical results are reported in Figure \ref{figure:adaptive_md_linear}---\ref{figure:movie}. These figures show that for both synthetic and real-world data, the proposed APM converges much faster than other algorithms: When APM finds approximate CE prices of desired precision, the mirror descent, additive t\^atonnement, and PDHG only achieve a much lower precision. In most cases, APM only uses $\frac14$ to $\frac15$ of the iterations needed by other algorithms to find the desired approximate CE prices, demonstrating its efficiency. PDGH is not reported in the quasi-linear setting because the involved EG program requires CCNH conditions, which are not satisfied by quasi-linear utilities \cite[Sec. 1]{liu2025pdhcg}. We note that the objective function value gap of APM iterates oscillates in the early stage of the iterations. This can be attributed to the non-monotonicity of Nesterov's acceleration (see, e.g., \cite[Sec. 6.1.2]{dAspremont2021AccelerationMethods}). 

 \subsection{Adaptive APM for Computing Exact CE}
We run the adaptive APM until the stopping criterion, i.e., the optimality conditions of Problem \eqref{eq:QL-exp}, is satisfied within a tolerance of $10^{-8}$. We evaluate the iteration number and CPU time used by the adaptive APM. As a comparison, we also record the CPU time of the Mosek solver to compute CE, which solves \eqref{eq:Unify1} to output a solution with a precision of $10^{-6}$. Additionally, we compute the radius $\Delta^*$ based on the recovered CE prices. The experiments are conducted using both synthetic data (with $20$ repeated trials) and real-world data collected from the movie rating dataset.
\begin{figure}[htbp]
    \centering
        \includegraphics[width=0.35\textwidth]{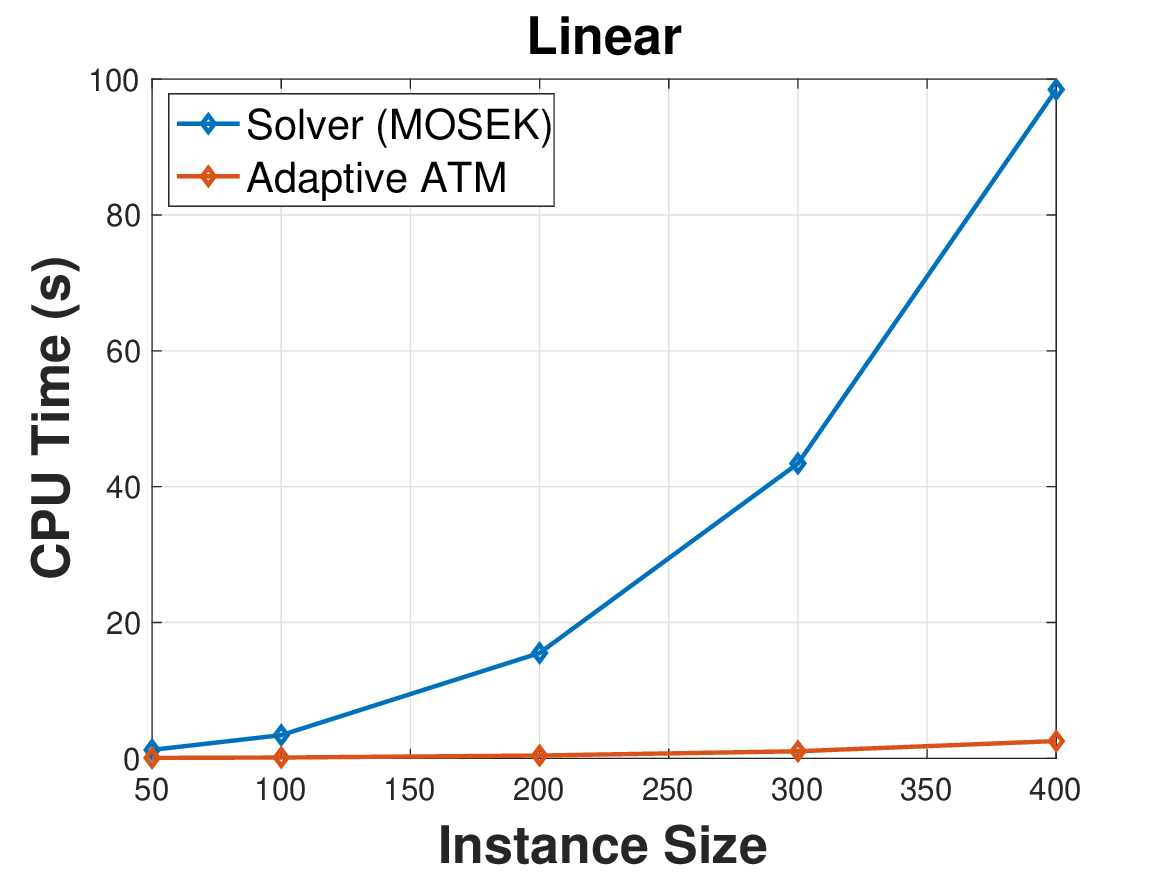} 
        \includegraphics[width=0.35\textwidth]{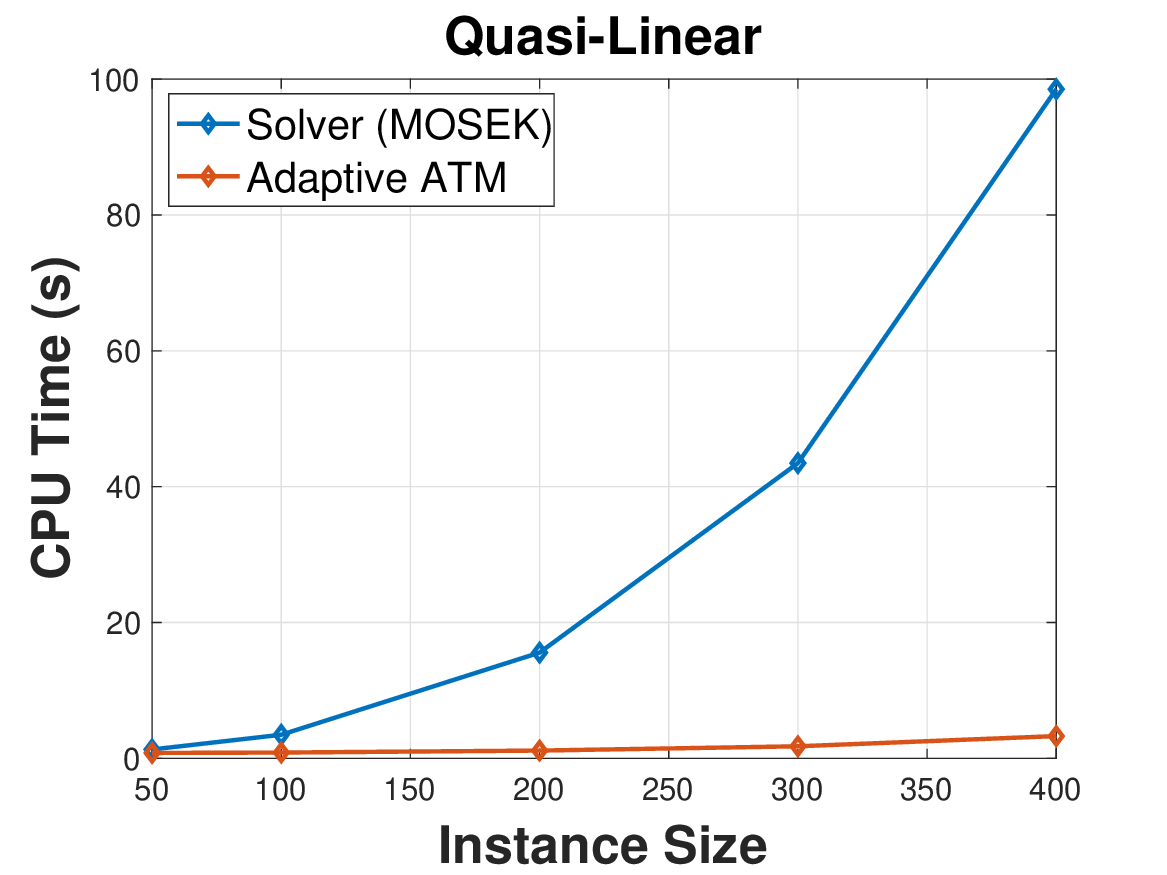} 
         \caption{Comparison of the CPU time of adaptive APM and Mosek solver for computing an exact CE with integer synthetic data. }
    \label{figure:exact_integer}
\end{figure}

\begin{table}[h]
\centering
\begin{tabular}{|c|c|c|c|c|c|c|}
\hline
\textbf{Data} & \textbf{$(n,m)$} & \textbf{Utility}& \textbf{Iteration} & \textbf{$\Delta^*$} & \textbf{CPU time (s)} &\textbf{Solver time (s)} \\ \hline
\multirow{8}*{Integer}       & \multirow{2}*{(50,50)} & Linear      & 2463               & 0.00131 & \textbf{0.0620}   &  1.2455    \\ \cline{3-7}
~       & ~             & Quasi-linear  & 3788       & 0.00015   & \textbf{0.0703}   &  1.1223    \\ \cline{2-7}
~       & \multirow{2}*{(100,100)}  & Linear   & 2925               & 0.00242   & \textbf{0.1186}    &  3.4148   \\ \cline{3-7}
~       & ~  & Quasi-linear   & 10219               & 0.00002 & \textbf{0.3302}    &  3.3930      \\ \cline{2-7}
~       & \multirow{2}*{(200,200)}  & Linear   & 10500 & 0.00012   & \textbf{0.4115}    &  15.5087    \\ \cline{3-7}
~       & ~  & Quasi-linear   & 4670              & 0.00450  & \textbf{0.1993}    &  16.0328    \\ \cline{2-7}
~       & \multirow{2}*{(300,300)}  & Linear   & 3296               & 0.00051    & \textbf{1.0407}    &  43.4084  \\ \cline{3-7}
~       & ~  & Quasi-linear   & 7600               & 0.00270    & \textbf{2.0301}    &  41.0377   \\ \cline{2-7}
~       & \multirow{2}*{(400,400)}  & Linear   & 9858               & 0.00032     & \textbf{2.5446}    &  98.4671 \\ \cline{3-7}
~       & ~  & Quasi-linear   & 7239               & 0.00334  & \textbf{2.3128}    &  102.9551     \\ \hline 
\multirow{8}*{Exp}       & \multirow{2}*{(50,50)} & Linear      & 1238               & 0.00272 & \textbf{0.0223}   &  1.2123    \\ \cline{3-7}
~       & ~             & Quasi-linear  & 7352       & 0.00005   & \textbf{0.1103}   &  0.9431    \\ \cline{2-7}
~       & \multirow{2}*{(100,100)}  & Linear   & 9937               & 0.00002   & \textbf{0.5103}    &  3.4121   \\ \cline{3-7}
~       & ~  & Quasi-linear   & 12130              & 0.00003 & \textbf{0.2702}    &  3.4235      \\ \cline{2-7}
~       & \multirow{2}*{(200,200)}  & Linear   & 60523 & 0.00001   & \textbf{2.1323}    &  15.1238    \\ \cline{3-7}
~       & ~  & Quasi-linear   & 2789              & 0.00023  & \textbf{0.1553}    &  16.4562    \\ \cline{2-7}
~       & \multirow{2}*{(300,300)}  & Linear   & 32391               & 0.00001    & \textbf{9.5784}    &  43.1021  \\ \cline{3-7}
~       & ~  & Quasi-linear   & 23123              & 0.00003    & \textbf{6.0347}    &  38.3302   \\ \cline{2-7}
~       & \multirow{2}*{(400,400)}  & Linear   & 29353               & 0.00004     & \textbf{8.4263}    &  102.3145 \\ \cline{3-7}
~       & ~  & Quasi-linear   & 47239               & 0.00003  & \textbf{12.5263}    &  105.9257     \\ \hline
\multirow{8}*{Logrnd}       & \multirow{2}*{(50,50)} & Linear      & 2990               & 0.00211 & \textbf{0.0551}   &  1.2353    \\ \cline{3-7}
~       & ~             & Quasi-linear  & 1278       & 0.00237   & \textbf{0.0302}   &  1.1201    \\ \cline{2-7}
~       & \multirow{2}*{(100,100)}  & Linear   & 3345               & 0.00312   & \textbf{0.1812}    &  3.5642   \\ \cline{3-7}
~       & ~  & Quasi-linear   & 3648               & 0.00212 & \textbf{0.2912}    &  3.4223      \\ \cline{2-7}
~       & \multirow{2}*{(200,200)}  & Linear   & 21324 & 0.00002   & \textbf{0.9243}    &  15.5890    \\ \cline{3-7}
~       & ~  & Quasi-linear   & 43679              & 0.00001  & \textbf{2.1345}    &  16.3213    \\ \cline{2-7}
~       & \multirow{2}*{(300,300)}  & Linear   & 44560               & 0.00002    & \textbf{14.1264}    &  42.9084  \\ \cline{3-7}
~       & ~  & Quasi-linear   & 23905               & 0.00004    & \textbf{9.5782}    &  41.0293   \\ \cline{2-7}
~       & \multirow{2}*{(400,400)}  & Linear   & 67390               & 0.00002     & \textbf{16.2356}    &  98.4131 \\ \cline{3-7}
~       & ~  & Quasi-linear   & 78924               & 0.00002  & \textbf{18.2317}    &  102.9953     \\ \hline
\multirow{2}*{\makecell{Movie\\ Rating} } &  \multirow{2}*{ (691,632)} & Linear & 17432 & $7.2\times10^{-5}$ & \textbf{2.5441}  & 272.3047\\  \cline{3-7}
~ &  ~ & Quasi-linear & 21328 & $2.4\times10^{-5}$ & \textbf{2.7332} & 267.1123\\ \hline
\end{tabular}
\vspace{1em}
\caption{Performance of the adaptive APM in computing an exact CE.}\label{table:exact_integer}
\end{table}
Figure \ref{figure:exact_integer} and Table \ref{table:exact_integer} show that the adaptive APM uses approximately $\frac{1}{100}$ to $\frac{1}{10}$ running time of the Mosek solver to compute CE in most of cases. Furthermore, as the instance size grows, the CPU time of the adaptive APM increases slightly, whereas that of the Mosek solver increases significantly. These results highlight the efficiency of adaptive APM. Additionally, we observe a roughly negative correlation between the iteration number and $\Delta^*$, consistent with our theory (i.e., Remark \ref{re:recovery3}). 

\section{Conclusion and Future Directions} \label{sec:conclusion}
In this paper, we proposed a unified strongly convex formulation for computing CE of linear and quasi-linear Fisher markets. Based on this formulation, we developed new price-adjustment processes to overcome the limitations of t\^atonnement methods. Specifically, by applying Nesterov's acceleration, we developed APM that predicts the future excess-supply to adjust prices. We proved that APM finds $\epsilon$-CE prices in $\tilde{\cO}(1/\sqrt{\epsilon})$ iterations, which significantly improves upon the iteration complexities of the t\^atonnement methods. Furthermore, we constructed a recovery oracle that maps approximate CE prices to exact CE prices at a low computational cost. By coupling this recovery oracle with price-adjustment processes, we derived adaptive price-adjustment methods and showed that they find CE in finite steps. Finally, we conducted numerical experiments to demonstrate the fast convergence of APM and the efficient recovery of CE for adaptive APM. 

Our developments suggest several directions for future research: (i) accelerating price-adjustment processes for other utilities; (ii) providing a conditional bound on the radius $\Delta^*$ for the recovery oracle; (iii) extending the recovery oracle to general market settings. Directions (i) and (iii) could contribute to the design of faster price-adjustment processes with finite-step convergence guarantees in general market settings. Direction (ii) may help develop a (conditional) polynomial-time complexity for adaptive APM to compute exact CE. We believe that exploring these directions will deepen our understanding of CE computation and advance its practical applications.

\bibliographystyle{plainnat}
\bibliography{ref}
\newpage
\appendix
The appendix is structured as follows. In Appendices \ref{appen:epsilon-CE} and \ref{appen:exact-CE}, we provide the missing proofs for Sec. \ref{sec:epsilon-CE} and \ref{sec:exact-CE}, respectively. In Appendix \ref{appen:test}, we show that the optimality test for Problem \eqref{eq:QL-exp} reduces to a max-flow problem and discuss its computational cost.

\section{Missing Proofs of Sec. \ref{sec:epsilon-CE}}\label{appen:epsilon-CE}
\subsection{Proof of Fact \ref{fact:f_delta}}
(i)\&(ii) Observe that the exponential sum $\sum_{j\in [m]}\exp(\mu_j)$ is $\exp({\bar{\mu}+\eta})$-smooth, $\exp({\underline{\mu}-\eta})$-strongly convex in the feasible region. Moreover, by \cite[p. 74]{boyd2004convex} and the composition rule, we see that the composition of linear and $\log$-sum-$\exp$ functions, i.e., $ \mu\mapsto\log(\sum_{j\in \{0\}\cup[m]}\exp({\frac{\log(v_{ij})-\mu_j}{\delta}}))$, is convex, with hessian smaller than $\frac{1}{\delta^2}\bI_m$.
It follows that the function $F_{\delta}$ is smooth with modulus $L=\exp({\bar{\mu}+\eta})+\|B\|_1/{\delta}$ and strongly convex with modulus $\sigma=\exp({\underline{\mu}-\eta})$. 

\noindent(iii) Notice that $0\geq\sum_{j \in [m] \cup \{ 0\}} \lambda_{ij}\log(\lambda_{ij})\geq-\log(m+1)$ for $\lambda_i\in\cD_{m+1}$, $i\in[n]$. We see from \eqref{eq:F_max_form} and \eqref{eq:F_delta_max} that $0\leq F_{\delta}-F\leq \delta\log(m+1) \|B\|_1$. 

\noindent(iv)  Let $\cJ_i(\mu)\coloneqq \argmax_{j\in\{0\}\cup[m]}\{\log(v_{ij})-\mu_j\}$ for $i\in[n]$, $\mu\in\R^m$,and recall that $\mu_0=\log(\alpha)$ is the parameter and $v_{i0}=1$. Direct computation gives \[\nabla F_{\delta}(\mu)\to d\coloneqq\left(\exp({\mu_1}),\ldots,\exp({\mu_m})\right)-\sum_{i\in[n]}\frac{B_i}{|\cJ_i(\mu)|}\sum_{j\in \cJ_i(\mu)}\e_j\in\partial F(\mu)\quad\text{ as }\quad\delta\to0.\]

\subsection{Proof of Theorem \ref{th:SAG}}\label{subsec:SAG}
Denote the optimal value of Problem \eqref{eq:box} (resp. \eqref{eq:box-smooth}) by $F^{*}$ (resp. $F^*_{\delta}$). Let $\hat{\mu}^{\delta}\coloneqq \argmin_{\mu\in\R^m}F_{\delta}(\mu)$.
We first prove the following lemma, which says that after the relaxation of the box constraints, the global minimizer of $F_{\delta}$, i.e., $\hat{\mu}^{\delta}$, is located in the new box $[(\underline{\mu}-\eta)\1_m,(\bar{\mu}+\eta)\1_m]$.
\begin{lemma}\label{le:stop2}
     Suppose that the parameters $\delta>0,\eta\geq0$ satisfy $2\delta\log(m+1)\|B\|_1\leq \exp({\underline{\mu}-\eta})\eta^2 $. Then, $\hat{\mu}^{\delta}$ is also the optimal solution of the box-constrained problem \eqref{eq:box-smooth}.
\end{lemma}
\begin{proof}
It suffices to show that $\underline{\mu}-\eta\leq\hat{\mu}^{\delta}_j\leq \bar{\mu}+\eta$, $j\in[m]$. Recall that $\mu^*\coloneqq \arg\min_{\mu\in\R^m}F(\mu)$ satisfies $\mu^*_{j}\in[\underline{\mu},\bar{\mu}]$, $j\in[m]$. We only need to prove the inequality $\|\hat{\mu}^{\delta}-\mu^*\|_2\leq \eta$.

    Let us prove $\|\hat{\mu}^{\delta}-\mu^*\|_2\leq \eta$ by contradiction. Suppose the converse that $\|\hat{\mu}^{\delta}-\mu^*\|_2> \eta$. Then, there exists $s\in(0,1)$ such that the vector $\mu^s\coloneqq s\hat{\mu}^{\delta}+(1-s)\mu^*$ satisfies $\|\mu^s-\mu^*\|_2=\eta$. It follows from the strong convexity of $F$ and $\mu^*= \arg\min_{\mu\in\R^m}F(\mu)$  that 
    \begin{equation}\label{eq:strict}
        F(\mu^s)<s F(\hat{\mu}^{\delta})+(1-s)F(\mu^*)<F(\hat{\mu}^{\delta}).
    \end{equation}
    Further, notice that $\mu^s_j\geq \underline{\mu}-\eta$, $\mu^*_j\geq\underline{\mu}$ for $j\in[m]$, and $F$ is $\exp({\underline{\mu}-\eta})$-strongly convex on $[(\underline{\mu}-\eta)\1_m,+\infty)$. We have
    \begin{equation}\label{eq:new_bound}
 \frac{\exp({\underline{\mu}-\eta})}{2}\|\mu^s-\mu^*\|_2^2 \leq F(\mu^s)-F(\mu^*) < F(\hat{\mu}^{\delta})-F(\mu^*),\end{equation}
 where the second inequality is due to \eqref{eq:strict}.
 
 On the other hand, by Fact \ref{fact:f_delta} (iii) and the definition $\hat{\mu}^{\delta}=\argmin_{\mu\in\R^m}F_{\delta}(\mu)$, we have
 \begin{equation}\label{eq:strict2}
     F(\hat{\mu}^{\delta})-F(\mu^*)\leq F_{\delta}(\hat{\mu}^{\delta})-F(\mu^*)\leq  F_{\delta}(\mu^*)-F(\mu^*)    \leq \delta\log(m+1)\|B\|_1.
 \end{equation}
 Combining \eqref{eq:strict} and \eqref{eq:strict2}, and recalling $\|\mu^s-\mu^*\|_2=\eta$, we see that
 \begin{equation}\label{eq:contra}
     \frac{\exp({\underline{\mu}-\eta})}{2} \eta^2< \delta\log(m+1)\|B\|_1, 
 \end{equation}
 which contradicts the condition $2\delta\log(m+1)\|B\|_1\leq {\exp({\underline{\mu}-\eta})} \eta^2 $. We complete the proof.
\end{proof}
By Lemma \ref{le:stop2} and the first-order optimality of unconstrained problems, we know that the optimal solution of \eqref{eq:box-smooth}, i.e., $\hat{\mu}^{\delta}$, satisfies $\nabla F_{\delta}(\hat{\mu}^{\delta})=0$. This, together with \citet[Theorem 2.1.10]{Nesterov1983AMF} and \citet[Theorem 2.1.5]{Nesterov1983AMF}, implies the following corollary, which serves as the basis of the stopping criterion of APM. 
\begin{lemma}\label{le:stop}
Suppose that the parameters $\delta,\eta$ in Problem \eqref{eq:box-smooth} satisfy $2\delta\log(m+1)\|B\|_1\leq {\exp({\underline{\mu}-\eta})} \eta^2 $.
    Then, for all $\mu\in\R^m$, the following hold:
    \begin{enumerate}[{\rm (i)}, itemsep=1.5pt, leftmargin=*]
        \item $F_{\delta}(\mu)-F^*_{\delta}\leq \frac{1}{2\sigma}\|\nabla F_{\delta}(\mu)\|^2_2$;
        \item $F_{\delta}(\mu)-F^*_{\delta}\geq \frac{1}{2L}\|\nabla F_{\delta}(\mu)\|^2_2$.
    \end{enumerate}
\end{lemma}
Lemma \ref{le:stop} ensures the validity of the stopping criterion of APM. Specifically, if the stopping criterion $\|\nabla F_{\delta}(\mu^{{t}})\|_2\leq\min\{\sigma\epsilon,\sqrt{\sigma\epsilon}\}$ holds for $t=\bar{t}\geq0$, then Lemma \ref{le:stop} (i) implies \[ F_{\delta}\left(\mu^{\bar{t}}\right)-F^*_{\delta}\leq \frac{\epsilon}{2}. \]
Observe that $F(\mu^{\bar{t}})-F_{\delta}(\mu^{\bar{t}})\leq0$ by Fact \ref{fact:f_delta} (iii); $F_{\delta}^{*}\leq F_{\delta}(\mu^*)$ since $\mu^*\in[\underline{\mu},\bar{\mu}]^m$ is a feasible point of \eqref{eq:box-smooth} while $F_{\delta}^{*}$ is its optimal value; and  $F^*=F(\mu^*)$ by definition.
We further have
\begin{equation}\label{eq:criterion}
    \begin{aligned}
F\left(\mu^{\bar{t}}\right)-F^{*} &= F\left(\mu^{\bar{t}}\right)
        -F_{\delta}\left(\mu^{\bar{t}}\right)+F_{\delta}\left(\mu^{\bar{t}}\right)-F_{\delta}^{*}+F_{\delta}^{*}-F^{*}    \\
        &\leq 0+\frac{\epsilon}{2}+F_{\delta}(\mu^*)-F(\mu^*)\\
        &\leq \frac{\epsilon}{2}+\delta\log(m+1)\|B\|_1\\
        &=\frac{\epsilon}{2}+\frac{\epsilon}{2}=\epsilon,
\end{aligned}
\end{equation}
where the second inequality is due to Fact \ref{fact:f_delta} (iii) and the parameter choice $\delta= \frac{\epsilon}{2\log(m+1)\|B\|_1}$.
        
Given \eqref{eq:criterion}, we conclude that the stopping criterion ensures that the output $\mu^{\bar{t}}$ recovers an $\epsilon$-CE price vector.
Let us then focus on the lower bound on $\bar{t}$ that guarantees the stopping criterion $\|\nabla F_{\delta}(\mu^{\bar{t}})\|_2\leq\min\{\sigma\epsilon,\sqrt{\sigma\epsilon}\}$. By Lemma \ref{le:stop} (ii), it suffices to have
\begin{equation}\label{eq:cri2}
     F_{\delta}\left(\mu^{\bar{t}}\right)-F^*_{\delta}\leq \min\left\{\frac{\sigma^2\epsilon^2}{2L},\frac{\sigma\epsilon}{2L}\right\}=\frac{\delta}{2(\delta\exp({\bar{\mu}+\eta})+\|B\|_1)}\cdot \min\left\{\sigma^2\epsilon^2, \sigma\epsilon \right\}=\cO(\epsilon^3).
\end{equation}
Notably, it still reduces to the function value gap $F_{\delta}-F^*_{\delta}$. As our APM can be viewed as Nesterov's acceleration method, i.e., \citet[Scheme 2.3.13]{nesterov2013introductory}, for the $\sigma$-strongly convex, $L$-smooth Problem \eqref{eq:box-smooth}, the following result applies.
\begin{lemma}{\rm(\cite[Theorem 2.3.6]{nesterov2013introductory}).}
    Let $f$ be an $L_f$-smooth $\sigma_f$-strongly convex function with $L_f,\sigma_f>0$. Let $q_f=\frac{\sigma_f}{L_f}\in(0,1)$ and $f^*=\min_{y\in\mathcal{Y}} f(y)$, where $\mathcal{Y}\subseteq\R^m$ is the feasible set. The iterates of Nesterov's acceleration method, denoted by $\{y^t\}_{t\geq0}$, satisfy
\[f\left(y^t\right)-f^{*} \leq\left(1-\sqrt{q_f}\right)^t\left(f\left(y^0\right)-f^{*}\right),\quad\forall\ t\geq0.\]\label{lm:nestrerov}
\end{lemma}
Before applying Lemma \ref{lm:nestrerov}, we estimate the condition number $q=\sigma/L$ for the function $F_{\delta}$, which is key to determining the convergence rate of APM. Let us define the constants 
\[c_1=F_{\delta}\left(\mu_0\right)-F_{\delta}^{*}, \quad c_2=\frac{ \exp({\underline{\mu}-\eta}) }{3 \log(m+1) \|B\|_1^{2}},\quad \text{and}\quad c_3=\frac13\exp({\underline{\mu}-\bar{\mu}-2\eta}).\]
For the estimate of $q$, there are two cases based on the range of $\epsilon$.
The first case is  $\epsilon\leq \exp({-\bar{\mu}-\eta})\log(m+1)\|B\|_1^{2}$. For this case, one has 
\begin{equation}\label{eq:q}
    \begin{aligned}
  q=  {\frac{\exp({\underline{\mu}-\eta})}{\exp({\bar{\mu}+\eta})+\frac{1}{\delta}\|B\|_1}} = 
{\frac{\exp({\underline{\mu}-\eta}){\epsilon}}{\exp({\bar{\mu}+\eta}){\epsilon}+{2\log(m+1)}\|B\|_1^{2}}}  \geq{\frac{\exp({\underline{\mu}-\eta}){\epsilon}}{{3\log(m+1)}\|B\|_1^{2}}} =c_2\epsilon.
    \end{aligned}
\end{equation}
\text{For the other case that}  $\epsilon> \exp({-\bar{\mu}-\eta})\log(m+1)\|B\|_1^{2}$, the monotonic increasing property of the function $y\mapsto\frac{y}{a_1y+a_2}$  ($a_1,a_2>0$) implies that
\begin{equation}\label{eq:q2}
\begin{aligned}
      &q= 
{\frac{\exp({\underline{\mu}-\eta}){\epsilon}}{\exp({\bar{\mu}+\eta})\epsilon+{2\log(m+1)}\|B\|_1^{2}}} \\
&\geq {\frac{\exp({\underline{\mu}-\eta})\cdot{\exp({-\bar{\mu}-\eta})\log(m+1)\|B\|_1^{2}}}{{3\log(m+1)}\|B\|_1^{2}}}\\
&= \frac13\exp({\underline{\mu}-\bar{\mu}-2\eta})=c_3.
\end{aligned}
\end{equation}
Now, applying Lemma \ref{lm:nestrerov}, we have
\begin{equation}\label{eq:nesterov}
    \begin{aligned}
        F_{\delta}\left(\mu^{t}\right)-F_{\delta}^{*} &\leq \left(1-\sqrt{q}\right)^t \cdot \left(F_{\delta}\left(\mu_0\right)-F_{\delta}^{*}\right)\\
        & = c_1\cdot \exp({t\log(1-\sqrt{q})})\\
        &\leq c_1\cdot \exp({-\sqrt{q}\cdot t}),\\
    \end{aligned}
\end{equation}
where the second inequality is due to the fact that $\log(1+y)\leq y$ for all $y\in (-1,\infty)$.

For the first case that $\epsilon\leq \exp({-\bar{\mu}-\eta})\log(m+1)\|B\|_1^{2}$, \eqref{eq:nesterov} and \eqref{eq:q} yield $F_{\delta}(\mu^{t})-F_{\delta}^{*}  \leq c_{1} \cdot \exp({-\sqrt{c_{2}\epsilon}\cdot t})$.
It follows that the desired inequality \eqref{eq:cri2} holds whenever 
\begin{equation}\label{eq:iteration_number}
    \bar{t}\geq\frac{-\log\left(\frac1{c_1}\min\left\{\frac{\sigma^2\epsilon^2}{2L},\frac{\sigma\epsilon}{2L}\right\}\right)}{\sqrt{c_2\epsilon}}=\tilde{\cO}\left(\frac{1}{\sqrt{\epsilon}}\right).
\end{equation}
For the other case that $\epsilon> \exp({-\bar{\mu}-\eta})\log(m+1)\|B\|_1^{2}$, \eqref{eq:nesterov} and \eqref{eq:q2} yield
$ F_{\delta}(\mu^{t})-F_{\delta}^{*}  \leq c_{1} \cdot \exp({-\sqrt{c_{3}}\cdot t})$.
In this case, the desired inequality \eqref{eq:cri2} holds whenever 
\[\bar{t}\geq\frac{-\log\left(\frac1{c_1}\min\left\{\frac{\sigma^2\epsilon^2}{2L},\frac{\sigma\epsilon}{2L}\right\}\right)}{\sqrt{c_3}}=\cO\left(\log\left(\frac{1}{\epsilon}\right)\right).\]
Combining the above bounds on $\bar{t}$ in two cases, we conclude that APM takes at most $\tilde{\cO}(\frac{1}{\sqrt{\epsilon}})$ iterations to find an $\epsilon$-CE price vector. Thus, we complete the proof of Theorem \ref{th:SAG}.

\begin{remark}[Parameters of APM]\label{re:parameter}
For some other parameter choices of APM, the above arguments (especially, \eqref{eq:nesterov}, \eqref{eq:criterion}, and the estimates on $q$) are also valid, and hence the result of Theorem \ref{th:SAG} still holds. For instance, one can set $\epsilon\in(0,+\infty)$ and $\delta=\min\{\epsilon, {\exp({\underline{\mu}-\eta})}\eta^2\}/(2\log(m+1)\|B\|_1)$.
\end{remark}

\subsection{Proof of Proposition \ref{pro:appriximate}}
The proof of Proposition \ref{pro:appriximate} is similar to that of \citet[Lemma 5.1]{chen2024computing}. 
To begin, let us introduce a technical lemma adapted from \citet[Lemma C.2]{chen2024computing} that will prove useful for the proof.
\begin{lemma}\label{le:tech}
    Consider $a_1,a_2,\ldots,a_m>0$ with $m>1$ and $\gamma>1+\log(m-1)$. The following inequality holds:
    \begin{equation*}
        \frac{\sum_{j\in[m]}a_j^{\gamma+1}}{\sum_{j\in[m]}a_j^{\gamma}}\geq \max_{j\in[m]}\{a_j\} \cdot\frac{1-\frac{1.3}{\gamma+1}}{(m-1)^{1/(\gamma+1)}}.
    \end{equation*}
\end{lemma}
Then, we present the following observations concerning the optimal utility in (B2).
\begin{lemma}\label{le:max}
	Consider a price vector $p\in\R^m_+$. Let $\bar{v}_i=\max_{x_i \in \R_{+}^{n}}\{u_{i}(x_i):\sum_{j\in[m]} p_jx_{ij} \leq B_i\}$ for $i\in[n]$. The following hold:
	\begin{enumerate}[{\rm (i)},leftmargin=*]
		\item For linear utility $u_i(x_i)=v_i^{\top}x_i$, we have $\bar{v}_i=B_i\max_{j\in[m]}\{v_{ij}/p_j\}$.
		\item For quasi-linear utility $u_i(x_i)=(v_i-p)^{\top}x_i$, we have
		$\bar{v}_i=B_i\max_{j\in\{0\}\cup[m]}\{{v_{ij}}/{p_j}\}-B_i$.
	\end{enumerate}
\end{lemma}
\begin{proof}[Proof of Lemma \ref{le:max}]
(i) Under linear utilities, for all $x_i\in\R^m_+$ satisfying $p^{\top}x_i\leq B_i$, we have
\begin{equation*}
	u_i(x_i)=v_i^{\top}x_i=\sum_{j\in[m]}\frac{v_{ij}}{p_j}p_jx_{ij}\leq \max_{j\in[m]}\left\{\frac{v_{ij}}{p_j}\right\}\sum_{j\in[m]} p_jx_{ij}\leq B_i \max_{j\in[m]}\left\{\frac{v_{ij}}{p_j}\right\}.
\end{equation*}
Clearly, the equality can be attained by choosing $x_i=B_i\e_j/p_j$ with $j\in\argmax_{j\in[m]}\left\{{v_{ij}}/{p_j}\right\}$. We conclude that $\bar{v}_i=B_i\max_{j\in[m]}\{v_{ij}/p_j\}$.

(ii) If $0\in\argmax_{j\in\{0\}\cup[m]}\left\{{v_{ij}}/{p_j}\right\}$, then by $v_{i0}=p_0=1$, we have $\max_{\{0\}\cup[m]}\left\{{v_{ij}}/{p_j}\right\}=1$ and further $p_j\geq v_{ij}$ for $j\in[m]$. It follows that $u_i(x_i)=(v_i-p)^{\top}x_i\leq0$ for all $x_i\in\R^m_+$. This yields
\begin{equation*}
\bar{v}_i=\max_{x_i\in\R^m_+} \left\{u_{i}(x_i):\sum_{j\in[m]} p_jx_{ij} \leq  B_i\right\}=0=B_i\max_{j\in\{0\}\cup[m]}\left\{\frac{v_{ij}}{p_j}\right\}-B_i.
\end{equation*} 
If $0\notin\argmax_{j\in\{0\}\cup[m]}\left\{{v_{ij}}/{p_j}\right\}$, then by $v_{i0}=1$ and $p_0=1$, we know that $\max_{j\in[m]}\left\{{v_{ij}}/{p_j}\right\}>1$ and hence $p_j< v_{ij}$ for $j\in\argmax_{j\in[m]} \left\{{v_{ij}}/{p_j}\right\}$. Then, $u_i(x_i)=(v_i-p)^{\top}x_i>0$ for some $x_i\in\R^m_+$. We further see that \[\max\limits_{x_i\in\R^m_+} \left\{u_{i}(x_i):\sum_{j\in[m]} p_jx_{ij} \leq  B_i\right\}>0,\] where the optimal solution $\bar{x}_i$ satisfies  \[ \bar{x}_{ij}>0 ~\text{ only if }~j\in \argmax_{j\in[m]} \left\{\frac{v_{ij}-p_j}{p_j}\right\}=\argmax_{j\in[m]} \left\{\frac{v_{ij}}{p_j}\right\};\quad p^{\top}\bar{x}_i=B_i.\]
Let $J_i=\argmax_{j\in[m]}\{{v_{ij}}/{p_j}\}$. We have
\[v_i^{\top}\bar{x}_i=\sum_{j\in J_i}\frac{v_{ij}}{p_j}\bar{x}_{ij}p_j=\sum_{j\in J_i}\bar{x}_{ij}p_j\cdot \max_{j\in[m]} \left\{\frac{v_{ij}}{p_j}\right\}=\sum_{j\in[m]}\bar{x}_{ij}p_j\cdot\max_{j\in[m]} \left\{\frac{v_{ij}}{p_j}\right\} =B_i\max_{j\in[m]} \left\{\frac{v_{ij}}{p_j}\right\}.\]
Hence, we have the following estimate for $\bar{v}_i$:
\begin{equation*}
\bar{v}_i=\max_{x_i\in\R^m_+} \left\{u_{i}(x_i):\sum_{j\in[m]} p_jx_{ij} \leq  B_i\right\}=v_i^{\top}\bar{x}_i-p^{\top}\bar{x}_i=B_i\max_{j\in[m]} \left\{\frac{v_{ij}}{p_j}\right\}-B_i.
\end{equation*}
Combining the above two cases, we complete the proof.
\end{proof}

With the above lemmas in hand, we are ready to verify the conditions (B1), (B2), and (B3). First, we show  (B1) holds via the definition of $x$:
    \[
    p^{\top}x_i=B_i\frac{\sum_{j\in [m]}\exp\left({\frac{\log(v_{ij})-\mu_j}{\delta}}\right)}{\sum_{j\in\{0\}\cup[m]}\exp\left({\frac{\log(v_{ij})-\mu_j}{\delta}}\right)}\leq B_i.
    \]
We then prove (B3). Recall $p_j=\exp(\mu_j)$, $j\in[m]$ and let 
\[d_{ij}=B_{i}\cdot \frac{\exp\left({\frac{\log(v_{ij})-\mu_j}{\delta}}\right)}{\sum\limits_{j\in\{0\}\cup[m]} \exp\left({\frac{\log(v_{ij})-\mu_j}{\delta}}\right) }; \qquad i\in[n], j\in[m].\] 
    By direct computation, we have $\nabla_jF_{\delta}(\mu)=p_j-\sum_{i\in[n]}d_{ij}$. This, together with the stopping criterion $\|\nabla F_{\delta}(\mu)\|_2\leq \min\{\sigma\epsilon,\sqrt{\sigma\epsilon} \}$ and the fact $\sigma=\exp({\underline{\mu}-\eta})\leq \exp(\mu_j)= p_j$, yields $|\nabla_jF_{\delta}(\mu)|\leq p_j\epsilon$ and further
    \[ \epsilon\geq \left|\frac{\nabla_jF_{\delta}(\mu)}{p_j}\right| =\left|1-\sum_{i\in[n]}\frac{d_{ij}}{p_j}\right|,\qquad\forall\ j\in[m].\] 
    By definition, we have $x_{ij}=d_{ij}/p_j$.
    It follows that $|\sum_{i\in[n]}x_{ij}-1|\leq\epsilon$, which accords with (B3).

It is left to prove (B2). We first consider the \textit{linear utilities}, i.e., $u_i(x_i)=v_i^{\top}x_i$ and $\alpha=+\infty$. By Lemma \ref{le:max} (i), to show (B2), it suffices to prove 
    \begin{equation}\label{eq:dijxj}
         u_i(x_i)=v_{i}^{\top}x_{i}\geq\left(1-\frac{2\epsilon}{\|B\|_1}\right)B_i \max_{j\in[m]}\left\{\frac{v_{ij}}{p_j}\right\}.
    \end{equation}
    Let $p_0:=\exp({\mu_0})$. We next estimate $v_i^{\top}x_i$:
    \begin{equation}\label{eq:dijxj2}
        \begin{aligned}
            v_i^{\top}x_i = \sum_{j\in[m]} v_{ij}\frac{d_{ij}}{p_j} 
             = B_i\sum_{j\in [m]} \frac{v_{ij}}{p_j}\cdot  \frac{\left(\frac{v_{ij}}{\exp(\mu_j)}\right)^{\frac1{\delta}} }{\sum_{j\in \{0\}\cup[m]}\left(\frac{v_{ij}}{\exp(\mu_j)}\right)^{\frac1{\delta}} }  =  B_i\sum_{j\in[m]}  \frac{\left(\frac{v_{ij}}{p_j}\right)^{\frac1{\delta}+1} }{\sum_{j\in\{0\}\cup[m]}\left(\frac{ v_{ij} }{p_j}\right)^{\frac1{\delta}} }.
        \end{aligned}
    \end{equation}
    Notice $p_0=\exp({\mu_0})=\alpha=+\infty$  for linear utilities. \eqref{eq:dijxj} automatically holds if $m=1$. Consider $m\geq2$ and apply Lemma \ref{le:tech}. We have
    \begin{equation}\label{eq:vx}
        \begin{aligned}
            v_i^{\top}x_i =B_i  \frac{\sum_{j\in[m]}\left(\frac{v_{ij}}{p_j}\right)^{\frac1{\delta}+1} }{\sum_{j\in [m]}\left(\frac{ v_{ij} }{p_j}\right)^{\frac1{\delta}} }  \geq B_{i} \max_{j\in[m]}\left\{\frac{v_{ij}}{p_j}\right\} \cdot\frac{1-\frac{1.3}{\frac{1}{\delta}+1}}{(m-1)^{1/(\frac{1}{\delta}+1)}} \geq  B_{i} \max_{j\in[m]}\left\{\frac{v_{ij}}{p_j}\right\}\cdot \frac{1-1.3\delta}{(m-1)^{\delta}}. 
        \end{aligned}
    \end{equation} 
    Here, the second inequality follows from the monotonic increasing property of the function $x\mapsto (1-1.3x^{-1})/(m-1)^{x^{-1}}$ on $(1.3,+\infty)$ and the inequality $1/\delta+1>1/\delta=2\log(m+1)\|B\|_1/\epsilon>1.3$ (recall $\epsilon\leq \log(m+1)\|B\|_1$).

     Then, let us estimate ${(1-1.3\delta)}/{(m-1)^{\delta}}$. By direct computation, we have
     \[
     \begin{aligned}
         (1.3+\log(m-1))\delta=\frac{1.3+\log(m-1)}{2\log(m+1)\|B\|_1}\epsilon\leq  &\left(\frac{1.3}{2\log(m+1)}+\frac{\log(m-1)}{2\log(m+1)}\right)\frac{\epsilon}{\|B\|_1} \\
         \leq& \left(\frac{1.3}{2\log(2)}+\frac12\right)\frac{\epsilon}{\|B\|_1} \leq \frac{2\epsilon}{\|B\|_1}. 
     \end{aligned}\]
    It follows from $1\geq\exp(-\delta\log(m-1))\geq 1-\delta\log(m-1)$ that 
    \begin{equation}\label{eq:vx2}
    \begin{aligned}
         \frac{2\epsilon}{\|B\|_1}&\geq1.3\delta+ \delta\log(m-1)\\
         &\geq {1.3\delta}\cdot \exp({-\delta\log(m-1)})+1-\exp({- \delta\log(m-1)})\\
         &=1-(m-1)^{-\delta}(1-1.3\delta).
    \end{aligned}
    \end{equation}
    Hence, we see that $(m-1)^{-\delta}(1-1.3\delta)\geq1-2\epsilon/\|B\|_1$. This, together with \eqref{eq:vx}, yields the desired \eqref{eq:dijxj}. Therefore, (B2) holds for linear utilities.
   
    We then prove (B2) under \textit{quasi-linear utilities}, where $p_0=\exp({\mu_0})=\alpha=1$. We first estimate the left-hand side of (B2). Using the definition of $d_{ij}$, $\exp({\frac{\log(v_{ij})-\mu_j}{\delta}})=(v_{ij}/p_j)^{1/\delta}$, and $v_{i0}=p_0=1$, we have
\[p^{\top}x_i=\sum_{j\in[m]}p_j\frac{d_{ij}}{p_j}=\sum_{j\in[m]}d_{ij}=B_i\frac{\sum_{j\in[m]}\left(\frac{ v_{ij} }{p_j}\right)^{\frac1{\delta}} }{\sum_{j\in\{0\}\cup[m]} \left(\frac{ v_{ij} }{p_j}\right)^{\frac1{\delta}}}=B_i-B_i\frac{1 }{1+\sum_{j\in[m]} \left(\frac{ v_{ij} }{p_j}\right)^{\frac1{\delta}}}.\]
This, together with \eqref{eq:dijxj2}, yields 
    \begin{equation}\label{eq:quasilinear}
        (v_i-p)^{\top}x_i+B_i=B_i  \frac{1+\sum_{j\in[m]}\left(\frac{v_{ij}}{p_j}\right)^{\frac1{\delta}+1} }{1+\sum_{j\in [m]}\left(\frac{ v_{ij} }{p_j}\right)^{\frac1{\delta}} }.
    \end{equation}
    We then estimate the right-hand side of (B2).
  By Lemma \ref{le:max} (ii), we have
  \begin{equation}\label{eq:quasilinear2}
      \max_{x^{\prime}_i\in\R^m_+} \left\{u_{i}(x^{\prime}_i)+B_i:\sum_{j\in[m]} p_jx^{\prime}_{ij} \leq  B_i\right\}=B_i\max_{j\in\{0\}\cup[m]} \left\{\frac{v_{ij}}{p_j}\right\}-B_i+B_i=B_i\max_{j\in\{0\}\cup[m]} \left\{\frac{v_{ij}}{p_j}\right\}.
  \end{equation}
Given the left-hand and right-hand side expressions of (B2), i.e., \eqref{eq:quasilinear} and \eqref{eq:quasilinear2}, the proof of (B2) reduces to that for linear utilities, especially the arguments of \eqref{eq:vx} and \eqref{eq:vx2}. We complete the proof.

\section{Missing Proofs of Sec. \ref{sec:exact-CE}}\label{appen:exact-CE}
\subsection{Proof of Lemma \ref{le:recover}}\label{appen:recover}
For each $j\in[m]$, the optimality condition \eqref{eq:optimality2} implies that at the optimal solution $\mu^*\in\R^m$, there exists an index $i\in[n]$ such that $\lambda_{ij}>0$. It follows that $j\in\cJ_{i}(\mu^*)=J^*_{i}$ by the definition of $\lambda_i$. Therefore, we have \[[m]\subseteq\bigcup_{i\in[n]}J^*_{i}=\bigcup_{l\in[s]}\tilde{J}^*_{l},\]
where the equality is due to the definition of $\tilde{J}^*_{l},l\in[s]$.

We then focus on showing the equivalence between the optimality condition \eqref{eq:optimality2} and equations \eqref{eq:mu*1}, \eqref{eq:mu*2} over $l\in[s]$. 
We first prove that \eqref{eq:mu*1} and \eqref{eq:mu*2} can be implied by \eqref{eq:optimality2}. 
  Observe that the definitions of $\cJ_i$ and $J^*_i$ ensure \eqref{eq:mu*1}. We focus on \eqref{eq:mu*2}. Note that \eqref{eq:optimality2} implies that under the condition $0\notin \tilde{J}^*_l$,
\[\exp(\mu_j)=\sum_{i\in[n]}\lambda_{ij},\quad\forall~j\in[m],\quad\text{ where }\quad \lambda_i\in\conv\{B_i\e_j:j\in J^*_i\}.\]
Summing up the above equality over $j\in\tilde{J}^*_l$ (noticing  $0\notin \tilde{J}^*_l$), we have
\begin{equation}\label{eq:J*l}
    \sum_{j\in \tilde{J}^*_l}\exp(\mu_j)= \sum_{i\in[n]}\sum_{j\in \tilde{J}^*_l}\lambda_{ij}.
\end{equation}
Recall that $\{J^*_i,i\in I^*_l\},l\in[s]$ is a connection class and $\tilde{J}_l^*=\bigcup_{i\in I^*_l}J^*_i$. By definition, we have $J_i^*\cap \tilde{J}_l^*=\emptyset$ for $i\notin I_l$ and $J^*_i\subseteq \tilde{J}^*_l$ for $i\in I^*_l$. Since $ \lambda_i\in\conv\{B_i\e_j:j\in J^*_i\}$, we see that $\lambda_{ij}=0$ for $i\notin I^*_l$, $j\in \tilde{J}_l^*$; and $\sum_{j\in \tilde{J}^*_l}\lambda_{ij}=B_i$ for $i\in I^*_l$. These, together with \eqref{eq:J*l}, yield
\[ \sum_{j\in \tilde{J}^*_l}\exp(\mu_j)= \sum_{i\in I^*_l}\sum_{j\in \tilde{J}^*_l}\lambda_{ij}=\sum_{i\in I^*_l}B_i, \]
which accords with \eqref{eq:mu*2}. We conclude that \eqref{eq:mu*1} and \eqref{eq:mu*2} are implied by \eqref{eq:optimality2}, and thus the solution set of \eqref{eq:mu*1} and \eqref{eq:mu*2} contains that of \eqref{eq:optimality2}, i.e., $\{\mu^*\}$.
 
The remaining task is to show that the solution of \eqref{eq:mu*1} and \eqref{eq:mu*2} is unique. Select an arbitrary index pair $(i^*,j^*)$ satisfying $j^*\in J_{i^*}^*$ and $i^*\in I^*_l$. By $\tilde{J}^*_l=\bigcup_{i\in I^*_l}J^*_i$ and the definition of connection class, for each $j\in\tilde{J}^*_l$, there is an index $i\in I^*_l$ such that $j\in J_i^*$; and indices $i_1,i_2,\ldots i_{c}\in I^*_l$ and $j_1,j_2,\ldots j_{c+1}\in\tilde{J}^*_l$, $c\in[|I^*_l|]$ such that
\[j_1\in J_{i^*}^*\cap J_{i_1}^*;\quad j_2\in J_{i_1}^*\cap J_{i_2}^*;\quad\ldots\quad, j_c\in J_{i_{c-1}}^*\cap J_{i_{c}}^*;\quad j_{c+1}\in J_{i_c}^*\cap J_{i}^*.\]
It follows that
\[j^*,j_1\in J_{i^*}^*;\quad j_1,j_2\in J_{i_1}^*;\quad\ldots\quad;\quad j_{c-1},j_{c}\in J_{i_{c-1}}^*;\quad j_{c},j_{c+1}\in J_{i_{c}}^*;\quad j_{c+1},j\in J^*_i.\]
This, together with \eqref{eq:mu*1}, implies  
\begin{equation}\label{eq:chain}
\begin{array}{c}
     \log(v_{i^*j^*})-\mu_{j^*}= \log(v_{i^*j_1})-\mu_{j_1};\\
      \log(v_{i_1j_1})-\mu_{j_1}= \log(v_{i_1j_2})-\mu_{j_2};\quad\ldots\quad; \quad\log(v_{i_cj_{c}})-\mu_{j_{c}}= \log(v_{i_cj_{c+1}})-\mu_{j_{c+1}};\\
        \log(v_{ij_{c+1}})-\mu_{j_{c+1}}= \log(v_{ij})-\mu_{j}.
\end{array}
\end{equation}
By \eqref{eq:chain}, we see that for every $\mu$ satisfying \eqref{eq:mu*1} and $j\in\tilde{J}^*_l$, there is a constant $a^{j^*}_j$, which can be computed through the addition and subtraction of $\log(v_{ij})$ for $i\in I^*_l$, $j\in \tilde{J}^*_l$, such that
\begin{equation}\label{eq:chain2}
    \mu_j=\mu_{j^*}+a^{j^*}_j,\quad \forall~j\in\tilde{J}_l^*.
\end{equation}
Clearly, $a^{j^*}_j$ is unique as we have shown that \eqref{eq:mu*1} and \eqref{eq:mu*2} are solvable with a solution $\mu^*$.

Now, we are ready to prove that the solution of equations \eqref{eq:mu*1} and \eqref{eq:mu*2} is unique and exactly computable. First, consider the case that $0\in\tilde{J}^*_l$, where the utilities must be quasi-linear and $\mu_0=\log(\alpha)=0$. Set $j^*=0$. Then, the solution coordinates $\mu_j,j\in \tilde{J}^*_l$ are uniquely implied by \eqref{eq:chain2}, i.e., 
\begin{equation}\label{eq:comp1}
    \mu_j=a^0_j,\quad  j\in\tilde{J}_l^*.
\end{equation}
For the case that $0\notin\tilde{J}^*_l$, substituting \eqref{eq:chain2} into \eqref{eq:mu*2}, we have
\[ \sum_{j\in\tilde{J}^*_l}\exp\left({\mu_{j^*}+a^{j^*}_j}\right)= \sum_{i\in I^*_l}B_i, \]
which yields a unique solution $\mu_{j^*}=\log(\sum_{i\in I^*_l}B_i)-\log(\sum_{j\in \tilde{J}^*_l}\exp({a^{j^*}_j}))$. It follows that
\begin{equation}\label{eq:comp2}
     \mu_j=\log\left(\sum_{i\in I^*_l}B_i\right)-\log\left(\sum_{j\in \tilde{J}^*_l}\exp\left({a^{j^*}_j}\right)\right)+a^{j^*}_j,\quad j\in\tilde{J}_l^*.
\end{equation}
Combining the two cases, we complete the proof.

\subsection{Proof of Lemma \ref{le:index}}
       We prove the equality from two directions: (i) $\cJ_i(\mu^*)\subseteq\{j\in\{0\}\cup[m]:\log(v_{ij})-\mu_j\geq h_i(\mu)-2r\}$; (ii) $\cJ_i(\mu^*)\supseteq\{j\in\{0\}\cup[m]:\log(v_{ij})-\mu_j\geq h_i(\mu)-2r\}$.

        \textbf{Direction (i):} Consider an arbitrary index $j\in\cJ_i(\mu^*)$. For $\mu\in\B(\mu^*,r)$, the triangle inequality yields
        \begin{equation*}
           \begin{aligned}
               |\log(v_{ij})-\mu_{j}-h_i(\mu)| & \leq |\log(v_{ij})-\mu_{j}^*-h_i(\mu^*) | + |\mu_{j}^*-\mu_j| + |h_{i}(\mu^*)-h_{i}(\mu)|\\
               & \leq 0+\|\mu^*-\mu\|_2+\|\mu^*-\mu\|_2\\
               &\leq 2r, \label{eq:direct1}
            \end{aligned}
       \end{equation*}
       where the second inequality is due to $j\in\cJ_i(\mu^*)$, $|\mu_{j}^*-\mu_j|\leq \|\mu^*-\mu\|_2$, and the $1$-Lipschitz continuity of the function $h_i$.
 It follows that $\log(v_{ij})-\mu_{j}\geq h_i(\mu)-2r$. This proves Direction (i).

          \textbf{Direction (ii):} Consider an arbitrary element $j$ of $ \{j\in\{0\}\cup[m]:\log(v_{ij})-\mu_j\geq h_i(\mu)-2r\}$. Notice that $\log(v_{ij})-\mu_j\leq h_i(\mu)$ by definition of $h_i$. We see that $-2r\leq\log(v_{ij})-\mu_j- h_i(\mu)\leq0$ and hence
          \[ |\log(v_{ij})-\mu_j- h_i(\mu)|\leq 2r.\]
          This, together with the triangle inequality, yields 
            \begin{equation}\label{eq:direct2}
           \begin{aligned}
              |\log(v_{ij})-\mu_{j}^*-h_i(\mu^*) |  & \leq  |\log(v_{ij})-\mu_{j}-h_i(\mu)|+ |\mu_j - \mu_{j}^*| + |h_{i}(\mu)-h_{i}(\mu^*)|\\
               & \leq 2r+\|\mu-\mu^*\|_2+\|\mu-\mu^*\|_2 \\
               & \leq 4r< \Delta^*,
           \end{aligned}
       \end{equation}
       where the second inequality uses $|\mu_{j}^*-\mu_j|\leq \|\mu^*-\mu\|_2$, $ \|\mu^*-\mu\|_2\leq r$, and the $1$-Lipschitz continuity of the function $h_i$; the third inequality is due to the condition $r<\Delta^*/4$.
       
By the definition of $\Delta^*$, i.e., \eqref{eq:def-delta}, if $j\notin \cJ_i(\mu^*)$, then one has $ h_i(\mu^*)-(\log(v_{ij})-\mu_{j}^*)\geq\Delta^*$, which contradicts \eqref{eq:direct2}. Hence, we have $j\in \cJ_i(\mu^*)$ for $j$ in $ \{j\in\{0\}\cup[m]:\log(v_{ij})-\mu_j\geq h_i(\mu)-2r\}$. This implies the desired $\{j\in\{0\}\cup[m]:\log(v_{ij})-\mu_j\geq h_i(\mu)-2r\}\subseteq \cJ_i(\mu^*)$. The proof is complete.

\subsection{Proof of Lemma \ref{le:class}}\label{appen:cost}
(i) For the sake of distinction, let $I_s^{\text{out}},\tilde{J}^{\text{out}}_s,s\in[s^{\text{out}}]$ (resp. $I_s ,\tilde{J}_s$) denote the outputs (resp. variables in process) of Algorithm \ref{al:class}. We use outer (resp. inner) iteration to refer to the operations from Line 3 to Line 15 (resp. Line 8 to Line 14), indexed by $s\in[s^{\text{out}}]$ (resp. $t\in[|\tilde{J}^{\text{out}}_s|]$).   Notice that the updates in Line 6, 10, 12, 13 ensure $\tilde{J}^{\text{out}}_s=\bigcup_{i\in I^{\text{out}}_s}J_i$. We focus on showing that $\{J_i:i\in I^{\text{out}}_s\},s\in[s^{\text{out}}]$ are the connection classes.

We first show that $\tilde{J}_s^{\text{out}}\cap \tilde{J}^{\text{out}}_{s^{\prime}}=\emptyset$ for $s\neq s^{\prime}$. Indeed, by the update of $\tilde{J}_s$ in Line 6, 12, 13, we have
\begin{equation}\label{eq:I_s}    \tilde{J}_s^{\text{out}}=J_{i^s}\cup\left(\bigcup_{t\in\left[\left|\tilde{J}_s^{\text{out}}\right|\right]}\bigcup_{i\in\cI_{new}(j^s_t)}J_i\right).\end{equation}
Observe that the $s$-th outer iteration ends with $\tilde{J}_s^{\text{out}}=J_{check}=\{j^s_t:t\in|\tilde{J}_s^{\text{out}}|\}$ and  $j^s_t\notin J_i$ for $i\in I^c$. We see that the $s$-th outer iteration ends with 
\begin{equation}\label{eq:empty}
	\left(\bigcup_{i\in I^c} J_i\right)\cap \tilde{J}^{\text{out}}_s=\emptyset.
\end{equation} 
In other words, for $s^{\prime}>s$, the $s^{\prime}$-th outer iteration begins with \eqref{eq:empty}.
 Note that the set $I^c$ at the beginning of the $s^{\prime}$-th ($s^{\prime}\geq2$) outer iteration satisfies $I^c\supseteq I^{\text{out}}_{s^{\prime}}$ and further $\bigcup_{i\in I^c} J_i\supseteq  \tilde{J}_{s^{\prime}}^{\text{out}}$. It follows that $\tilde{J}^{\text{out}}_{s^{\prime}}\cap\tilde{J}^{\text{out}}_s=\emptyset$ for all $s^{\prime}>s$. Clearly, this can be restated as $\tilde{J}^{\text{out}}_{s^{\prime}}\cap\tilde{J}^{\text{out}}_s=\emptyset$ for $s^{\prime}\neq s$.

We then prove that $I^{\text{out}}_s=\{i\in[n]:J_i\sim J_{i^s}\}$, $s\in[s^{\text{out}}]$. Note that the update of $\tilde{J}_s$ in Line 6, 10, 12, 13 ensures that $\tilde{J}_s=\bigcup_{i\in I_s} J_i$ at Line 8. Hence, in Line 9 of the $s$-th outer iteration, $j^s_t\in J_{\bar{i}}$ for some $\bar{i}\in I_s$, and then  
$J_i\sim J_{\bar{i}}$ for all $i\in\cI_{new}(j^*)$ by definition. This, together with the initialization $I_s=\{i^s\}$ and the update in Line 10, yields \[I^{\text{out}}_s\subseteq\{i\in[n]:J_i\sim J_{i^s}\}.\]
On the other hand, recalling that $\tilde{J}_s^{\text{out}}\cap J^{\text{out}}_{s^{\prime}}=\emptyset$ for $s\neq s^{\prime}$ and  $\tilde{J}^{\text{out}}_s=\bigcup_{i\in I^{\text{out}}_s}J_i$, we see that 
$J_i\cap J_{i^{\prime}}=\emptyset$ for all $i\in I^{\text{out}}_s,i^{\prime}\in I^{\text{out}}_{s^{\prime}}$ with $s\neq s^{\prime}$. It follows that $J_{i}\nsim J_{i^{\prime}}$ for $i\in  I^{\text{out}}_s$, $i^{\prime}\notin  I^{\text{out}}_s$ by checking the definition. Hence, for $i^s\in I^{\text{out}}_s$, we have
\[\{i\in[n]:J_i\sim J_{i^s}\}\subseteq I^{\text{out}}_s .\]
Combined with  $I^{\text{out}}_s\subseteq\{i\in[n]:J_i\sim J_{i^s}\}$, it yields $I^{\text{out}}_s=\{i\in[n]:J_i\sim J_{i^s}\}$, $s\in[s^{\text{out}}]$. We conclude that $\{J_i:i\in I^{\text{out}}_s\}$ forms a connection class for all $s\in[s^{\text{out}}]$.

(ii) 
We analyze the computational cost of the operations for each line.

\textit{Line 2, 4, 5}: 
Clearly, the one-time cost is at most $\cO(n)$, and they are repeated for at most $s^{\text{out}}=\cO(n)$ times. Hence, their total computational cost is at most $\cO(n^2)$. 

\textit{Line 6}: The one-time cost is at most $\cO(m)$, it is repeated for at most $s^{\text{out}}=\cO(n)$ times. Hence, the total computational cost is at most $\cO(mn)$. 

\textit{Line 7, 8, 13, 14}: The operations include merging and substracting the subsets of $\{0\}\cup[m]$, and checking whether $J_{check}\subsetneqq \tilde{J}_s$. Hence, the one-time cost is at most $\cO(m)$. Since the inner iteration is repeated for $|\tilde{J}^{\text{out}}_s|$ times in the $s$-th outer iteration, their repetition number is
\begin{equation}\label{eq:repeat1}
\sum_{s\in[s^{\text{out}}]}\left|\tilde{J}^{\text{out}}_s\right|=\left|\bigcup_{s\in[s^{\text{out}}]}\tilde{J}^{\text{out}}_s\right|\leq|\{0\}\cup[m]|=m+1,
\end{equation}
where the first equality is due to $\tilde{J}^{\text{out}}_s\cap \tilde{J}^{\text{out}}_{s^{\prime}}=\emptyset$ for $s\neq s^{\prime}$; and the second equality is due to the fact that $\tilde{J}^{\text{out}}_s\subseteq \{0\}\cup[m]$ for all $s\in[s^{\text{out}}]$.

We see that the total computational cost of Line 7, 8, 13, and 14 is at most $\cO(m^2)$.

\textit{Line 9, 10, 11}: The operations include merging and substracting the subsets of $[n]$, and checking whether\footnote{One can first spend a time complexity of $\cO(mn)$ to define a matrix $M\in\R^{n\times(m+1)}$, where $M_{i,j+1}=1$ if $j\in J_i$ and $M_{i,j+1}=0$ otherwise. Then, checking whether $J_i,i\in [n]$ contains $j^*$ is equivalent to checking whether $M_{i,j^*+1}=1$ for $i\in[n]$, where the cost is $\cO(n)$ for each $j^*\in\{0\}\cup[m]$.} $J_i,i\in [n]$ contains $j^s_t$. We see that the one-time cost is at most $\cO(n)$. The repetition number of Line 9, 10, and 11 is the same as Line 7, 8, 13, and 14, given by \eqref{eq:repeat1}. Therefore, the total computational cost of Line 9, 10, and 11 is at most $\cO(mn)$.

\textit{Line 12}: 
The operation of Line 12 merges $|\cI_{new}(j^s_t)|$ subsets of $\{0\}\cup[m]$ at a one-time cost of $\cO(m\cdot|\cI_{new}(j^s_t)|)$. We see that
for the $s$-th outer iteration, the cost of operation of Line 12 is
\begin{equation}\label{eq:repeat2}
\cO\left(m\sum_{t\in\left[\left|\tilde{J}^{\text{out}}_s\right|\right]}|\cI_{new}(j^s_t)|\right)=\cO\left(m\left|\bigcup_{t\in\left[\left|\tilde{J}^{\text{out}}_s\right|\right]} \cI_{new}(j^s_t)\right|\right)=\cO\left(m\left|I^{\text{out}}_s \right|\right).
\end{equation}
Here, the first equality is due to $\cI_{new}(j^s_t)\cap\cI_{new}(j^s_{t^{\prime}})=\emptyset$ for $j^s_t\neq j^s_{t^{\prime}}$, which is ensured by Line 11 and definition of $\cI_{new}(j^s_t)$ in Line 9; and the second equality uses the update in Line 10, which ensures that $\bigcup_{j^s_t\in\tilde{J}_s} \cI_{new}(j^s_t)\subseteq {I}^{\text{out}}_s$. 

The total cost of the operation of Line 12 is a summation of \eqref{eq:repeat2} over $s\in[s^{\text{out}}]$, which is equal to
\begin{equation}\label{eq:repeat3}
    \cO\left(m\sum_{s\in[s^{\text{out}}]} \left|I^{\text{out}}_s \right|\right)=\cO\left(m\left|\bigcup_{s\in[s^{\text{out}}]}I^{\text{out}}_s \right|\right)=\cO(mn).
\end{equation}
Similar to \eqref{eq:repeat1}, the first equality of \eqref{eq:repeat3} is due to the fact that $I^{\text{out}}_s\cap I^{\text{out}}_{s^{\prime}}=\emptyset$ for $s\neq s^{\prime}$; the second equality uses $I^{\text{out}}_s\subseteq[n]$, $s\in[s^{\text{out}}]$. 

Combining the above analysis, we conclude that the computational cost of $\cE$ is at most $\cO(m^2+mn+n^2)=\cO((m+n)^2)$.
 \subsection{Proof of Lemma \ref{le:solution}}
  We derive the expression of $\mu^*$ based on the index sets output by the classification procedure:
  \[\left\{I^*_l,\tilde{J}^*_l,i^l,j^l_t,\cI_{new}(j^l_t):t\in[|\tilde{J}^*_l|],l\in[s]\right\}=\cE\left(\left\{J^*_i,i\in[n]\right\}\right),\] and then show that $\mu^*$ equals the output $\mu=\cP(\cE(\{J^*_i,i\in[n]\}))$.
  
We first show $\mu^*_j=\mu^*_{j^l_1}+a_j$ for $j\in \tilde{J}^*_l$ (for any $l\in[s]$). Notice $\tilde{J}^*_l=J^*_{i^l}\cup(\bigcup_{t\in[|\tilde{J}^*_l|]}\bigcup_{i\in\cI_{new}(j^l_t)}J^*_i)$. We  prove the equality $\mu^*_j=\mu^*_{j^l_1}+a_j$ by induction. 

Let us begin with $\mu^*_j$, $j\in J^*_{i^l}$.
  By Line 8 of  Algorithm \ref{al:class}, we see that $j^l_1\in J^*_{i^l}$. It follows from the definition of $J^*_{i^l}$ that $\log(v_{i^lj^l_1})-\mu^*_{j^l_1}=\log(v_{i^lj})-\mu^*_{j}$ for $j\in J^*_{i^l}$. This, together with $a_j=\log(v_{i^lj})-\log(v_{i^lj^l_1})$ in Line 3 of Algorithm \ref{al:solution}, yields
     \begin{equation}\label{eq:mujl1}
         \mu^*_{j}=\mu^*_{j^l_1}+\log(v_{i^lj})-\log(v_{i^lj^l_1})=\mu^*_{j^l_1}+a_j,\quad\forall\ j\in J^*_{i^l}.
     \end{equation}
     Now, we assume that $\mu^*_j=\mu^*_{j^l_1}+a_j$ holds for $j\in J_{done}^{t-1}$ with some $t\in [|\tilde{J}^*_l|]$, where
     \[J_{done}^{t-1}\coloneqq J^*_{i^l}\cup\left(\bigcup_{\tau\in[t-1]}\bigcup_{i\in\cI_{new}(j^l_\tau)}J^*_i\right);\]
     and prove that $\mu^*_j=\mu^*_{j^l_1}+a_j$ for  $j\in J^*_{i^l}\cup(\bigcup_{\tau\in[t]}\bigcup_{i\in\cI_{new}(j^l_{\tau})}J^*_i)$ to finish the induction. Given the assumption, it suffices to prove that $\mu^*_j=\mu^*_{j^l_1}+a_j$ for $j\in  \bigcup_{i\in\cI_{new}(j^l_{t})}J^*_i\setminus J_{done}^{t-1}$ to complete the induction.
     
     Notice that $j^l_t\in J_i^*$ for all $i\in\cI_{new}(j^l_t)$. We have
     \[\mu^*_{j}=\mu^*_{j^l_t}+\log(v_{ij})-\log(v_{ij^l_t}),\qquad\forall\ j\in  \bigcup_{i\in\cI_{new}(j^l_{t})}J^*_i,~i\in\cI_{new}(j^l_t).\] 
     On the other hand, since $j^l_t\in J^*_{i^l}\cup(\bigcup_{\tau\in[t-1]}\bigcup_{i\in\cI_{new}(j^l_{\tau})}J^*_{i})$, one has $\mu^*_{j^l_t}=\mu^*_{j^l_1}+a_{j^l_t}$ by assumption. It follows that
     \begin{equation*}\label{eq:mujl2}
     \begin{aligned}
          \mu^*_{j}&=\mu^*_{j^l_t}+\log(v_{ij})-\log(v_{ij^l_t})\\ 
          &=\mu^*_{j^l_1}+a_{j^l_t}+\log(v_{ij})-\log(v_{ij^l_t})\\
          &=\mu^*_{j^l_1}+a_j,
          \qquad\qquad\forall~  j\in  \bigcup_{i\in\cI_{new}(j^l_{t})}J^*_i\setminus J_{done}^{t-1},
     \end{aligned}
     \end{equation*}
     where the third equality is due to the definition of $a_j$ in Line 5 of Algorithm \ref{al:solution}.
     Therefore, we complete the induction, ensuring that $\mu^*_j=\mu^*_{j^l_1}+a_j$ for all $j\in \tilde{J}^*_l$.
     
     Next, we compute $\mu^*_j$ for $j\in\tilde{J}^*_l$. We consider two cases: (i)  $0\in\tilde{J}^*_l$; (ii) $0\notin\tilde{J}^*_l$.

    (i) For the case that $0\in\tilde{J}^*_l$, it must be quasi-linear setting, where $\alpha=1$. We have $0=\log(\alpha)=\mu^*_0=\mu^*_{j^l_1}+a_0$, and hence $\mu^*_{j^l_1}=-a_0$, which accords with $\mu_{j^l_1}=-a_0$ given by Line 9 of Algorithm \ref{al:solution}. Further, since the output $\mu$ also satisfies $\mu_j=\mu_{j^l_1}+a_j$, $j\in \tilde{J}^*_l$, we see that $\mu^*_j=\mu_j$ for $j\in \tilde{J}^*_l$. 

     (ii) For the case that $0\notin\tilde{J}^*_l$,
     substitute $\mu^*_j=\mu^*_{j^l_1}+a_j$, $j\in \tilde{J}^*_l$ into the necessary condition \eqref{eq:mu*2}.  We obtain
     \[\mu^*_{j^l_1}=\log\left(\sum_{i\in I^*_l}B_i\right)-\log\left(\sum_{j\in\tilde{J}^*_l}\exp({a_j})\right),\] which accord with $\mu_{j^l_1}$ given by Line 12 of Algorithm \ref{al:solution}. As $\mu_j=\mu_{j^l_1}+a_j$, $j\in \tilde{J}^*_l$, we have  $\mu^*_j=\mu_j$ for $j\in \tilde{J}^*_l$. 
     
     Combining the two cases, we see that $\mu^*_j=\mu_j$ for all $j\in \bigcup_{l\in[s]}\tilde{J}^*_l$. Then, the desired $\mu^*=\mu$ follows from $[m]\subseteq\bigcup_{l\in[s]}\tilde{J}^*_l$ given by Lemma \ref{le:recover}. 
     We conclude that the output vector is the optimal solution of \eqref{eq:QL-exp} by Lemma \ref{le:recover}, i.e., $\cP(\cE(\{J^*_i,i\in[n]\}))=\mu^*$.

     The dominant computation cost is due to the operations of selecting $j\in J_i\setminus J_{done}$ for $i\in\cI_{new}(j^l_t)$ in Line 5, and computing $J_{done}\cup(\bigcup_{i\in \cI_{new}(j^l_t)}J_i)$ in Line 6. Observe that their cost is of the same order as Line 12 of Algorithm \ref{al:class}, which is $\cO(mn)$ by the arguments of \eqref{eq:repeat2}, \eqref{eq:repeat3} in the proof of Lemma \ref{le:class} (ii). We conclude that the computational cost of Algorithm \ref{al:solution} is $\cO(mn)$. 
\subsection{Proof of Remark \ref{re:recovery}}\label{appen:re_recovery}
It suffices to show that the price vector $p^*$ and allocation $x^*$ satisfy (E1), (E2), and (E3) of Definition \ref{def:ME}. First, the definition of $x^*$ and the first equation of \eqref{eq:test_stationarity} directly yields (E1):
\[\sum_{j\in[m]}p^*_jx^*_{ij}=\sum_{j\in[m]}\lambda^*_{ij}\leq B_i, \]
For (E3), we notice $p^*_j=\exp({\mu^*_j})$ and use the second equation of \eqref{eq:test_stationarity} to compute
\[\sum_{i\in[n]}x^*_{ij}=\frac{\sum_{i\in[n]}\lambda^*_{ij}}{p^*_j}=\frac{\exp({\mu^*_j})}{p^*_j}=1;\quad\forall\ j\in[m].\]
It is left to prove (E2). Notice that $x^*_{ij}=\lambda^*_{ij}/p^*_j$ and $p^*_j>0$. Using $\lambda^*_{ij}=0$ for $j\notin\cJ_i(\mu^*)$, we see that $x^*_{ij}>0$ only if $j\notin\cJ_i(\mu^*)$. By the definition of $\cJ_i(\mu^*)$, we further have
\[x^*_{ij}>0\quad \text{ only if }\quad j\in\argmax_{j\in\{0\}\cup[m]}\left\{\log(v_{ij})-\mu^*_j\right\}=\argmax_{j\in\{0\}\cup[m]}\left\{\frac{v_{ij}}{p^*_j}\right\}.\]
Then, we consider linear and quasi-linear utilities, respectively. First, for \textit{linear utilities}, we have 
\[u_i(x^*_i)=\sum_{j\in[m]}v_{ij}x^*_{ij}=\sum_{j\in[m]}\frac{v_{ij}}{p^*_j}\lambda^*_{ij}=\sum_{j\in[m]}\lambda^*_{ij}\max_{j\in\{0\}\cup[m]}\left\{\frac{v_{ij}}{p^*_j}\right\}=B_i\max_{j\in\{0\}\cup[m]}\left\{\frac{v_{ij}}{p^*_j}\right\}. \]
Here, the last equality is due to the first equation of \eqref{eq:test_stationarity} and $0\notin \cJ_i(\mu^*)$, $\lambda_{i0}=0$ for linear utilities ($p^*_0=+\infty$). The above equation, together with Lemma \ref{le:max} (i), implies (E2) for linear utilities.

Next, we consider \textit{quasi-linear utilities}. 

\textbf{Case (i)}: $0\in\cJ_i(\mu^*)=\argmax_{j\in\{0\}\cup[m]}\left\{{v_{ij}}/{p^*_j}\right\}$.

 Recall $v_{i0}=1$ and $p_0=1$. We know that (i) $v_{ij}\leq p^*_j$ for $j\in\{0\}\cup[m]$; (ii) $x^*_{ij}>0$ and $j\in\cJ_i(\mu^*)$ only if $v_{ij}=p^*_j$. It follows that 
 \[u_i(x^*_i)=\sum_{j\in[m]}(v_{ij}-p^*_j)x^*_{ij}=0= \max_{x_i\in\R^m_+} \left\{u_{i}(x_i):\sum_{j\in[m]} p^*_jx_{ij} \leq  B_i\right\}, \]
 where the last equality is due to Lemma \ref{le:max} (ii). This proves (E2) for Case (i) of quasi-linear utilities.

\textbf{Case (ii)}: $0\notin\cJ_i(\mu^*)=\argmax_{j\in\{0\}\cup[m]}\left\{{v_{ij}}/{p^*_j}\right\}$.

In this case, we have $\lambda^*_{i0}=0$ and hence $\sum_{j\in[m]}\lambda^*_{ij}=B_i$ by the first equation of \eqref{eq:test_stationarity}. We obtain
\[u_i(x^*_i)=\sum_{j\in[m]}(v_{ij}-p^*_j)x^*_{ij}=\sum_{j\in[m]}\left(\frac{v_{ij}}{p^*_j}-1\right)\lambda^*_{ij}. \]
Also, by $\lambda^*_{ij}=0$ for $j\notin\cJ_i(\mu^*)$ and ${v_{ij}}/{p^*_j}=\max_{j\in\{0\}\cup[m]}\{{v_{ij}}/{p^*_j}\}$ for $j\in\cJ_i(\mu^*)$, we further have
\[u_i(x^*_i)=\sum_{j\in[m]}\left(\frac{v_{ij}}{p^*_j}-1\right)\lambda^*_{ij}=\sum_{j\in[m]}\lambda^*_{ij}\max_{j\in\{0\}\cup[m]}\left\{\frac{v_{ij}}{p^*_j}-1\right\}=B_i\max_{j\in\{0\}\cup[m]}\left\{\frac{v_{ij}}{p^*_j}-1\right\}. \]
This, along with Lemma \ref{le:max} (ii), implies (E2) for Case (ii) of quasi-linear utilities. We complete the proof.

  \section{Test of Optimality}\label{appen:test}
In this section, we consider the test of the optimality of Problem \eqref{eq:QL-exp}, which is used for the stopping criterion of the adaptive APM in Sec. \ref{sec:exact-CE}. Due to the convexity of the objective function $F$, it suffices to check whether $0\in\partial F(\mu)$, or equivalently the feasibility of \eqref{eq:optimality2} w.r.t. $\lambda\in\R^{n\times m}$. Further, by introducing $\lambda_{i0}\coloneqq B_i-\sum_{j\in[m]}\lambda_{ij}$, the test can be further formulated as checking the feasibility of the following linear system w.r.t. $\lambda\in\R^{n\times(m+1)}$:
\begin{equation}\label{eq:test_mu}
     \begin{array}{rlrll}
        \sum\limits_{j\in\{0\}\cup[m]}\lambda_{ij}&=B_i,\quad&\forall\  i\in[n]; \qquad\quad   \sum\limits_{i\in [n]}\lambda_{ij}&=\exp(\mu_j),\quad &\forall\ j\in[m];  \\
      \lambda_{ij}&\geq0,\quad &\forall\ j\in\{0\}\cup[m],i\in[n];\quad  \qquad\lambda_{ij}&=0,\quad &\forall\ j\notin \cJ_i(\mu),i\in[n].\\
    \end{array}
\end{equation}
We demonstrate that testing the feasibility of \eqref{eq:test_mu} is equivalent to solving a max-flow problem. Specifically, let us define the vertex set $V$ and the edge set $E$ by
\[V\coloneqq \left\{s;\ g_j,j\in[m];\ b_i,i\in[n];\ t\right\};\]
\[E\coloneqq\left\{(s,g_j),j\in[m];\quad (g_j,b_i),j\in\cJ_i(\mu),i\in[n];\quad (b_i,t),i\in[n]\right\}.\]
Let the capacities of the edges $(s,g_j)$ and $(b_i,t)$ be  $\exp(\mu_j)$ and $B_i$, respectively. Further, let the capacities of the edges $(g_j,b_i)$ be $B_i$ for $j\in\cJ_i(\mu)$, $i\in[n]$. For $i\in[n]$, $j\in[m]$, let  $q_j$, $p_i$, and $\lambda_{ij}$ denote the flows of $(s,g_j)$, $(b_i,t)$, and $(g_j,b_i)$, respectively. In particular, we set $\lambda_{ij}\equiv0$ for $j\notin \cJ_i(\mu)$, which is equivalent to the fact that the edge $(g_j,b_i)$ does not exist. Then, we obtain the following max-flow problem in optimization form, where the redundant constraints $\lambda_{ij}\leq B_i$, $i\in[n]$, $j\in\cJ_i(\mu)$ are omitted.
\begin{equation}\label{eq:maxflow}
    \begin{array}{cll}
         \max\limits_{q_i,p_j,\lambda_{ij}\geq0} &\quad \sum\limits_{i\in[n]}q_j& \\
         \text{subject to }  &\quad q_j\leq \exp(\mu_j),\quad\sum\limits_{i\in[n]}\lambda_{ij}=q_j,\qquad&\forall\ j\in[m]\\  
         &\quad p_i\leq B_i, \quad\sum\limits_{j\in\cJ_i(\mu)}\lambda_{ij}=p_i,\qquad&\forall\ i\in[n]\\
         &\quad \lambda_{ij}=0, \qquad&\forall\ j\notin\cJ_i(\mu).
    \end{array}
\end{equation}
\begin{proposition}[Max-flow Formulation for Optimality Test]\label{pro:maxflow}
    The linear system \eqref{eq:test_mu} is feasible if and only if the optimal value of the max-flow problem \eqref{eq:maxflow} equals $\sum_{j\in[m]}\exp(\mu_j)$.
\end{proposition}
Given Proposition \ref{pro:maxflow}, to test the optimality of $F$ at $\mu\in\R^m$, we only need to solve the max-flow problem \eqref{eq:maxflow} and check whether its optimal value equals $\sum_{j\in[m]}\exp(\mu_j)$. Hence, the involved computational cost is bounded by the cost of solving a max-flow problem with the vertex set $V$ and edge set $E$, which has been extensively studied in the literature. Efficient methods include the push–relabel algorithm \cite{goldberg1988new}, the algorithm of \citet{orlin2013max}, and the high-probability algorithm by \citet{chen2022maximum}, which have a time complexity of $\cO(|V|^3)=\cO((m+n)^3)$, $\cO(|V|\cdot|E|)=\cO((m+n)mn)$, and $\cO(|E|^{1+o(1)}L)=\cO((mn)^{1+o(1)}L)$, respectively. Here, $L$ is the bit-length of the input data (i.e., edge capacities).
Note that computing an $\epsilon$-CE ($\epsilon\leq1$) through price-adjustment methods requires a least time complexity of $\cO(mn\cdot\min\{m,n\})$\footnote{This complexity can be deduced from two facts: (i) The cost of each iteration of price-adjustment methods is $\cO(mn)$; (ii) The iteration number needed to compute an $\epsilon$-CE ($\epsilon\leq1$) is at least $\cO(\min\{m,n\})$; see, e.g., \eqref{eq:iteration_number} for APM.}. The optimality tests would not significantly increase the total computational cost of the adaptive price-adjustment methods presented in Sec. \ref{sec:exact-CE}.
\begin{proof}[Proof of Proposition \ref{pro:maxflow}]
First, suppose that the linear system \eqref{eq:test_mu} is feasible. We show that the optimal value of \eqref{eq:maxflow}, denoted by $\bar{v}$, equals $\sum_{j\in[m]}\exp(\mu_j)$. To see this, let $\lambda^*\in\R^{n\times(m+1)}$ be a solution of \eqref{eq:test_mu} and $\bar\lambda\in\R^{n\times m}$ be defined by $\bar\lambda_{ij}=\lambda^*_{ij}$ for $i\in[n]$, $j\in[m]$. Then, we have \[\sum_{i\in[n]}\bar{\lambda}_{ij}=\exp(\mu_j); \quad\sum_{j\in\cJ_i(\mu)}\bar{\lambda}_{ij}\leq B_i;\quad \bar{\lambda}_{ij}=0\quad \forall\ j\notin\cJ_i(\mu).\] Clearly, the pair $(\bar{q},\bar{p},\bar{\lambda})$ with $\bar{q}_j=\sum_{i\in[n]}\bar{\lambda}_{ij}=\exp(\mu_j)$ and $\bar{p}_i=\sum_{j\in\cJ_i(\mu)}\bar{\lambda}_{ij}\leq B_i$ is feasible for the max-flow problem \eqref{eq:maxflow}. It follows a lower bound on the optimal value of \eqref{eq:maxflow}:
\[\bar{v}\geq\sum_{i\in[n]}\bar{q}_j=\sum_{j\in[m]}\exp(\mu_j).\]
On the other hand, by the constraints $q_j\leq \exp(\mu_j)$, $j\in[m]$, it is direct to see $\bar{v}\leq\sum_{j\in[m]}\exp(\mu_j)$. We conclude that $\bar{v}=\sum_{i\in[n]}\bar{q}_j=\sum_{j\in[m]}\exp(\mu_j)$. 

The remaining task is to show that the equality $\bar{v}=\sum_{j\in[m]}\exp(\mu_j)$ implies the feasibility of the linear system \eqref{eq:test_mu}. Given $\bar{v}=\sum_{j\in[m]}\exp(\mu_j)$, we let the pair $(\bar{q},\bar{p},\bar{\lambda})$ be the optimal solution of the max-flow problem \eqref{eq:maxflow}. Then, we have \[\sum_{j\in[m]}\bar{q}_j=\bar{v}=\sum_{j\in[m]}\exp(\mu_j).\] This, together with the constraints $\bar{q}_j\leq \exp({\bar{\mu}_j})$, $j\in[m]$, yields
$\bar{q}_j=\exp(\mu_j)$. Using other constraints of \eqref{eq:maxflow}, we further see that $\bar{\lambda}_{ij}\geq0$ for $i\in[n]$, $j\in[m]$ and
\[\sum_{i\in[n]}\bar{\lambda}_{ij}=\bar{q}_j=\exp(\mu_j),\quad\forall\ j\in[m];\quad  \sum_{j\in[m]}\bar{\lambda}_{ij}=\bar{p}_i\leq B_i,\quad\forall\ i\in[n];\quad\bar{\lambda}_{ij}=0\quad\forall\ j\notin\cJ_i(\mu),i\in[n].\]
Now, define $\lambda^*\in\R^{n\times(m+1)}$ by $\lambda^*_{ij}=\bar{\lambda}_{ij}$ for $i\in[n]$, $j\in[m]$ and $\lambda^*_{i0}=B_i-\bar{p}_i$. Then, the above conditions on $\bar\lambda$ guarantee that $\lambda^*$  is a feasible solution for \eqref{eq:test_mu}. We complete the proof.
\end{proof}

\end{document}